\documentclass{amsart}

\usepackage{graphics}

\newcommand{\Mod}[1]{\ (\mathrm{mod}\ #1)}

\usepackage[T1]{fontenc}    

\usepackage{amsthm}
\usepackage{amsbsy,amsmath,amssymb,amscd,amsfonts}
\usepackage{hyperref}
\usepackage{url}            
\usepackage{booktabs}       
\usepackage{nicefrac}       
\usepackage{microtype}      

\usepackage{graphicx,float,latexsym,color}
\usepackage[font={small,it}]{caption}
\usepackage{subcaption}

\usepackage{makecell}

\usepackage[dvipsnames]{xcolor}

\newtheorem{theorem}{Theorem}
\newtheorem{observation}{Observation}
\newtheorem{proposition}{Proposition}

\newtheorem{remark}{Remark}
\newtheorem{corollary}{Corollary}

\hypersetup{
    pdftoolbar=true,        
    pdfmenubar=true,        
    pdffitwindow=false,     
    pdfstartview={FitH},    
    colorlinks=true,       
    linkcolor=OliveGreen,          
    citecolor=blue,        
    filecolor=black,      
    urlcolor=red           
}

\usepackage{lineno}

\title[Invariants of Self-Intersected N-Periodics]{Invariants of Self-Intersected\\N-Periodics in the Elliptic Billiard}
\author{Ronaldo Garcia}
\author{Dan Reznik}
\date{October 2020}

\begin{document}

\maketitle

\begin{abstract}
We study self-intersected N-periodics in the elliptic billiard, describing new facts about their geometry (e.g., self-intersected 4-periodics have vertices concyclic with the foci). We also check if some invariants listed in ``Eighty New Invariants of N-Periodics in the Elliptic Billiard'' (2020), \textit{arXiv:2004.12497} remain invariant in the self-intersected case. Toward that end, we derive explicit expressions for many low-N simple and self-intersected cases, including their caustics (for N=5 and 7 these depend on roots of 6th- and 12th-degree polynomials, respectively). We identify two special cases (one simple, one self-intersected) where a quantity prescribed to be invariant is actually variable.

\end{abstract}

\section{Introduction}
\label{sec:intro}
A point mass bouncing elastically in the interior of an ellipse is know as the {\em elliptic billiard}; see Figure~\ref{fig:n4n5}. Two quantities are conserved: (i) energy and (ii) product of the angular momenta
about the foci \cite{berry81-conserved,dacosta2015-billiard-gravity}. The latter implies every trajectory segment is tangent to a virtual confocal ellipse, known as the ``caustic''; still equivalently, a quantity known as Joachimsthal's constant $J$ is conserved \cite{sergei91}; see Appendix~\ref{app:invariants}. 

The elliptic billiard is a special case of Poncelet's porism \cite{dragovic11}, i.e., if a trajectory closes with N segments (i.e., a proper  is chosen), any other point on the outer ellipse can initiate a new N-periodic trajectory.

\begin{figure}
    \centering
    \includegraphics[width=\textwidth]{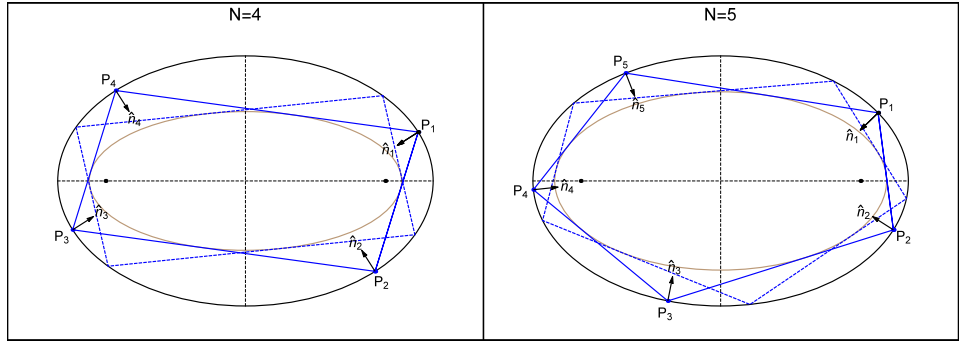}
    \caption{Elliptic Billiard (black) 4- and 5-periodics (blue). Every trajectory vertex $P_i$ (resp. segment $P_i P_{i+1}$) is bisected by the local normal $\hat{n}_i$ (resp. tangent to the confocal caustic, brown). A second, equi-perimeter member of each family is also shown (dashed blue). \href{https://youtu.be/Y3q35DObfZU}{Video}.}
    \label{fig:n4n5}
\end{figure}

Special to a confocal pair of ellipses is that the normal to a point on the outer ellipse bisects the two tangents to the inner one \cite{sergei91}. Since a Poncelet trajectory consists of a sequence of chords tangent to the caustic, if it closes after $N$ bounces, it must be the case that the normal at every vertex bisects the two tangents, i.e., any closed billiard trajectory will be N-periodic, i.e., a retracing sequence of elastic bounces.


The elliptic billiard is integrable. In fact it is conjectured as the only integrable planar billiard \cite{kaloshin2018}. A direct consequence is that the perimeter $L$ is conserved over said 1d family of N-periodics \cite{sergei91}. 


We have been probing the elliptic billiard for ``new'' cartesian properties, e.g., interesting manifestations of of constant $L$ and $J$ \cite{reznik2020-intelligencer,reznik2020-invariants}. A few examples include: (i) the sum of N-periodic angle cosines; (ii) the ratio of outer (aka. tangential) polygon (see Figure~\ref{fig:inner-outer}) to the N-periodic's for odd N. These were soon elegantly proved using tools of algebraic and differential geometry \cite{akopyan2020-invariants,bialy2020-invariants,caliz2020-area-product}. For the $N=3$ case some invariants have been explicitly derived \cite{garcia2020-new-properties}.

\begin{figure}
    \centering
    \includegraphics[width=.66\textwidth]{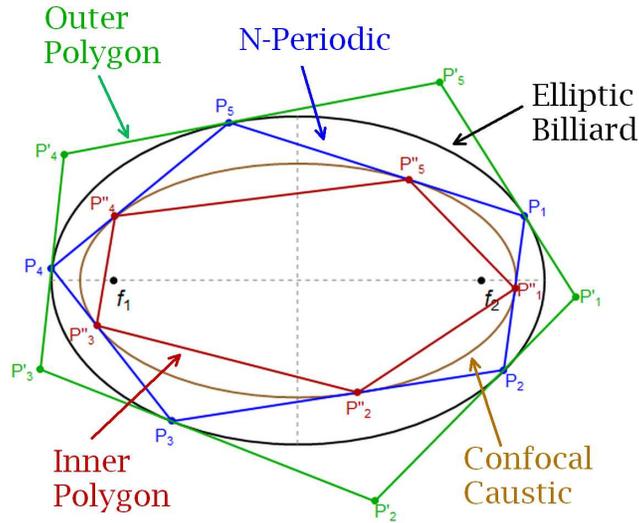}
    \caption{The N-periodic (blue), is associated with an outer (green) and an inner (red) polygons. The former's sides are tangent to the billiard (black) at each N-periodic vertex; the latter's vertices are the tangency points of N-periodics sides to the confocal caustic (brown). \href{https://youtu.be/PRkhrUNTXd8}{Video}}
    \label{fig:inner-outer}
\end{figure}

Experimentation with additional polygons such as pedals, antipedals, focus-inversives, etc., derived from N-periodics has increased the total number of invariants to 80+ \cite{reznik2020-invariants}. 

Addressed here for the first time is whether properties and invariants which hold for simple N-periodics also hold for their self-intersected versions; see Figures~\ref{fig:n7-three}. Birkhoff provides a method to compute the number of possible Poncelet N-periodics, simple or not \cite{birkhoff66}. For example, for N=5,6,7,8 there are 1,2,2,3 distinct self-intersected closed trajectories, respectively. 
\subsection*{Main contributions:}

\begin{itemize}
    \item Section~\ref{sec:n4-si}: The vertices of self-intersected 4-periodics and their outer polygon are concyclic with the foci on two distinct variable-radius circles.
    \item Section~\ref{sec:invariants}: for N=3,4,6,8, we derive expressions for selected invariants listed in \cite{reznik2020-invariants} including a few self-intersecting cases.
    \item In Appendix~\ref{app:5-periodic} (resp.  \ref{app:7-periodic}) we show that the caustic semi-axes for N=5 (N=7) periodics, simple or self-intersected, are roots of degree-6 (resp. degree-12) polynomials. There are two valid cases for $N=5$ and 3 for $N=7$; see Figure~\ref{fig:n7-three}.
    \item For at least two cases (Observations~\ref{obs:n4-804} and \ref{obs:n6ii-804}), an invariant that holds everywhere does not hold there, though the reasons are not yet clear.
\end{itemize}

\begin{figure}
    \centering
    \includegraphics[width=\textwidth]{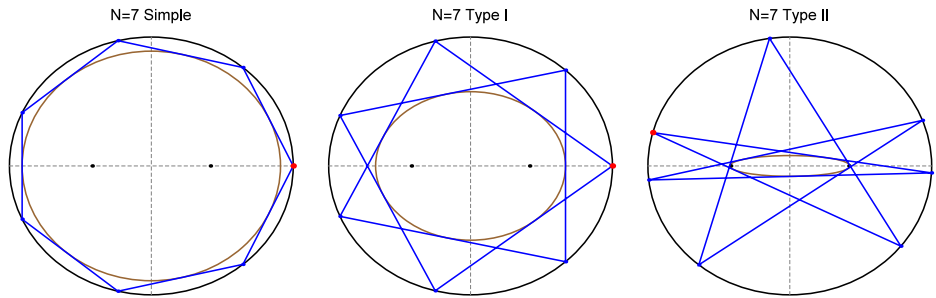}
    \caption{Three types of 7-periodics in an $a/b=1.1$ elliptic billiard, namely: (i) non-intersecting, (ii) self-intersected type I (turning number 2), (iii) self-intersected  type II (turning number 3).
    \href{https://youtu.be/yzBG8rgPUP4}{Video 1}, \href{https://youtu.be/BRQ39O9ogNE}{Video 2}}
    \label{fig:n7-three}
\end{figure}

Since many phenomena are best understood in motion, a link to a YouTube animation is provided in the caption of most figures. Table~\ref{tab:playlist} in Section~\ref{sec:videos} compiles all videos mentioned and a few extra ones. 

In Appendix~\ref{app:invariants} we review some classical identities of the elliptic billiard. In Appendix~\ref{app:vertices-caustics} we provide explicit expressions for vertices and caustics of $N=3,4,5,6,7,8$ families in both simple and self-intersected configurations. In Appendix~\ref{app:symbols} we list all symbols used.

\section{Preliminaries}

Throughout this article we assume the elliptic billiard is the ellipse:

\begin{equation}
\label{eqn:billiard-f}
f(x,y)=\left(\frac{x}{a}\right)^2+\left(\frac{y}{b}\right)^2=1,\;\;\;a>b>0.
\end{equation}

\subsection*{A word about our proof method}

We omit most proofs as they have been produced by a consistent process, namely: (i) using the expressions in Appendix~\ref{app:vertices-caustics}, find the vertices an axis-symmetric N-periodic, i.e., whose first vertex $P_1=(a,0)$; (ii) obtain a symbolic expression for the invariant of interest, (iii) simplify it (both human intervention and CAS), and finally (iv) verify its validity numerically over several N-periodic configurations and elliptic billiard aspect ratios $a/b$.


\section{Properties of Self-Intersected 4-Periodics}
\label{sec:n4-si}
The family of non-self-intersected (simple) billiard 4-periodics is comprised of parallelograms \cite{connes07}. In this section consider self-intersected 4-periodics whose caustic is a confocal hyperbola; see Figure~\ref{fig:bowtie-upright}. 
We start deriving simple facts about them and then proceed to certain elegant properties. 


\begin{proposition}
The perimeter $L$ of the self-intersected 4-periodic is given by:

\begin{equation}
L=\frac{4 a^2}{c},\;\;\;\mbox{with}\;\;\;c^2=a^2-b^2. 
\label{per-n4}
\end{equation}

\end{proposition}

\begin{proof}
Since perimeter is constant, use as the $N=4$ candidate the centrally-symmetric one, Figure~\ref{fig:bowtie-upright} (right). Its upper-right vertex $P_1=(x_1,y_1)$ is such that it reflects a vertical ray toward $-P_1$, and this yields:

\[
    P_1 = (x_1,y_1) = \left[\frac{a \sqrt{a^2-2 b^2}}{b c},\frac{b}{c}\right]
\]

\noindent Since $P_2=-P_1$ its perimeter is $L=2(|2P1|+2y_1)$ and this can be simplified to \eqref{per-n4},  invariant over the family.
\end{proof}

\noindent with $a/b{\geq}\sqrt{2}$. At $a/b=\sqrt{2}$ the family is a straight line from top to bottom vertex of the EB, Figure~\ref{fig:bowtie-upright} (left).

\begin{figure}
    \centering
    \includegraphics[width=\textwidth]{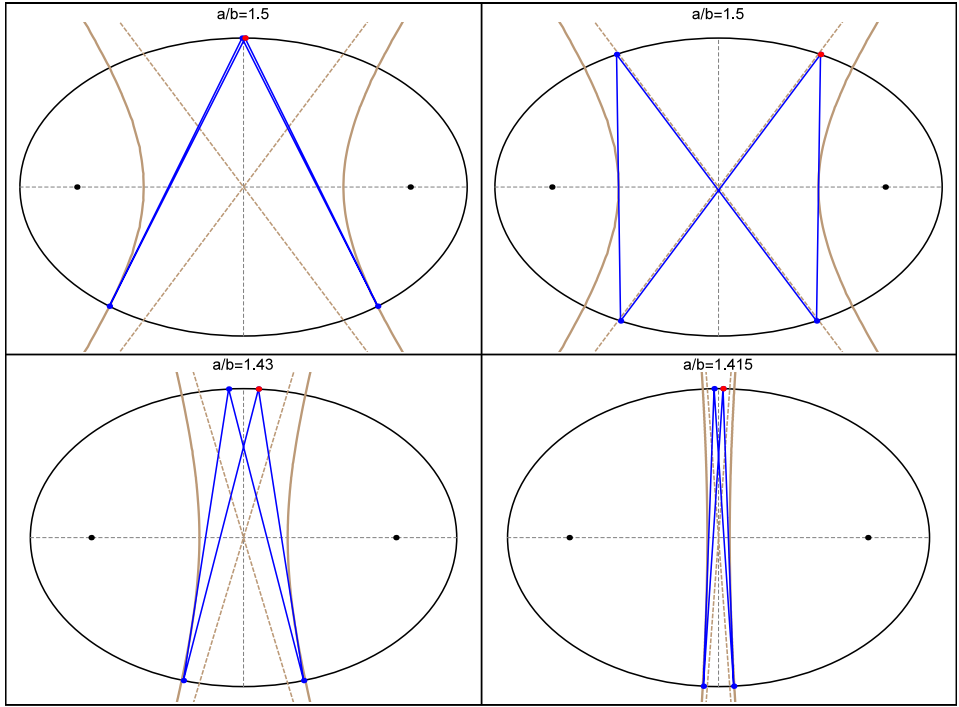}
    \caption{\textbf{Top Row:} two configurations of a self-intersected 4-periodic (blue) in an $a/b=1.5$ elliptic billiard (black). Its segments are tangent to a confocal hyperbolic caustic (brown). The left picture shows the trajectory near its ``doubled up'' configuration, whereas the right one shows an almost symmetric configuration where the diagonals coincide with the caustic's asymptotes (dashed brown). In the exactly symmetric configuration the side segments become vertical.  \textbf{Bottom row}: as the aspect ratio of the billiard decreases, the family is squeezed into an ever narrower space between the approaching branches of the caustic. At $a/b=\sqrt{2}$ they degenerate into a common vertical line at which point the family ceases to exist. \href{https://youtu.be/cCYxN7ueGV4}{Video}}
    \label{fig:bowtie-upright}
\end{figure}

\begin{observation}
At $a/b=\sqrt{1+\sqrt{2}}{\simeq}1.55377$ the two self-intersecting segments of the bowtie do so at right-angles.
\end{observation}

\begin{observation}
At $a/b{\simeq}1.55529$ the perimeter of the bowtie equal that of the EB.
\end{observation}

Referring to Figure~\ref{fig:n4-inv}:

\begin{proposition}
The $N=4$ self-intersected family has zero signed orbit area and zero sum of signed cosines, i.e., both are invariant. The same two facts are true for its outer polygon.
\end{proposition}

\begin{proof}
This stems from the fact all self-intersected 4-periodics are symmetric with respect to the EB's minor axis.
\end{proof}

\noindent Referring to Figure~\ref{fig:n4-inv}, as in Appendix~\ref{app:4-periodic}, let vertex $P_1$ of the self-intersected 4-periodic be parametrized as $P_1(u)=\left[a u,b\sqrt{1-u^2}\right]$, with $|u|{\leq}\frac{a}{c^2}
\sqrt {{a}^{2}-2\,{b}^{2}}$. Then:

\begin{figure}
    \centering
    \includegraphics[width=.8\textwidth]{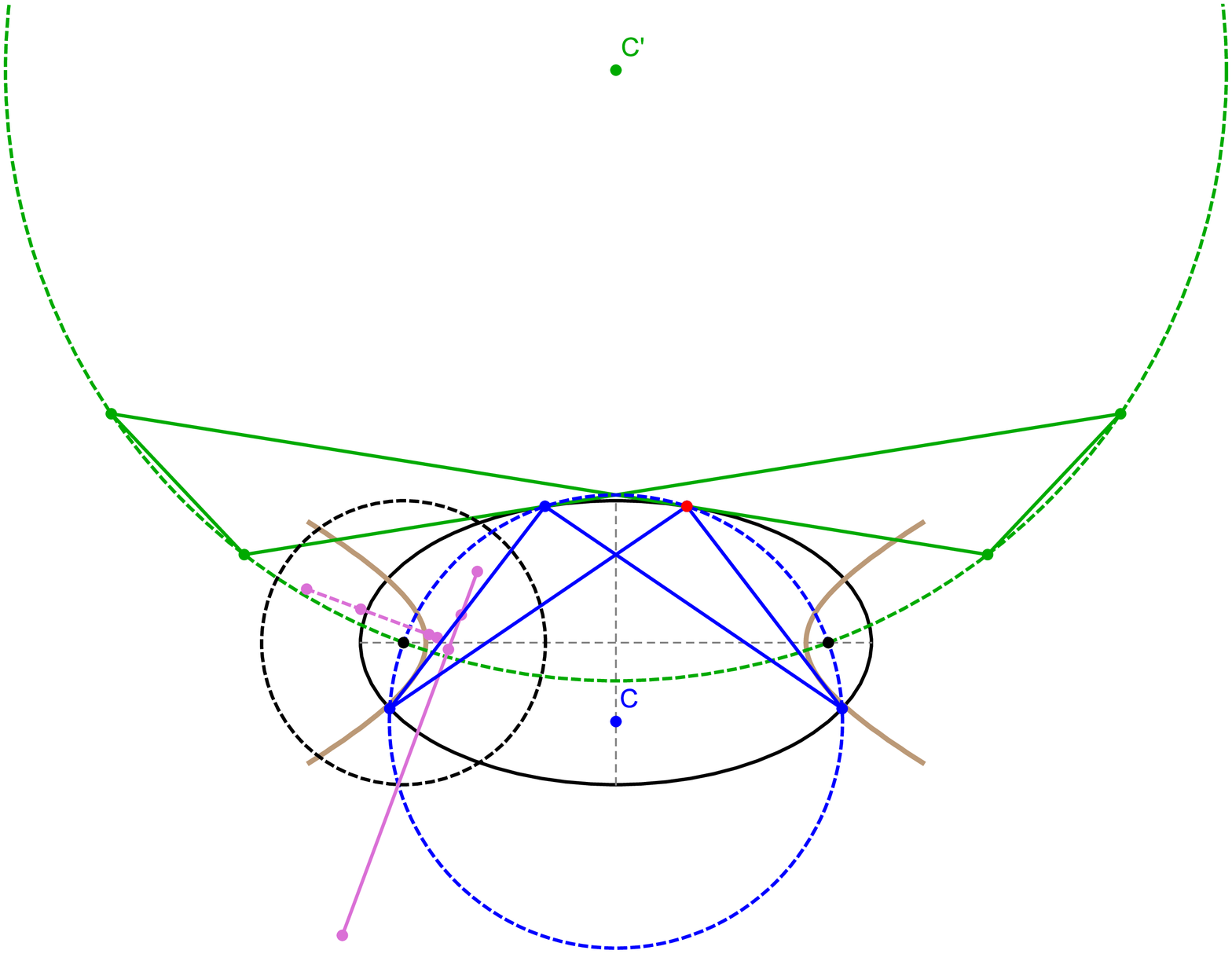}
    \caption{The vertices of the self-intersected 4-periodic (blue) are concyclic with the billiard foci on a circle (dashed blue) centered on $C$. The inversive polygon (pink segment) with respect to a unit circle $C^\dagger$ (dashed black) centered on the left focus degenerates to a segment along the radical axis of the two circles. The vertices of the outer polygon (green) are also concyclic with the foci on a distinct circle (dashed green) centered on $C'$. Therefore the outer's inversive polygon (dotted pink) is also a segment along the radical axis of this circle with $C^\dagger$. Note the two radical axes are dynamically perpendicular.  \href{https://youtu.be/207Ta31Pl9I}{Video 1}, \href{https://youtu.be/4g-JBshX10U}{Video 2}}
    \label{fig:n4-inv}
\end{figure}

\begin{theorem}
The four vertices of the self-intersected 4-periodic (resp. outer polygon) are concyclic with the two foci of the elliptic billiard, on a circle $\mathcal{C}$ of variable radius $R$ (resp. $R'$) whose center $C$ (resp. $C'$) lies on the y axis. These are given by:
\begin{align*}
    C=&\left[0,  \frac{c^{2}u^{2}-a^{2}+2\,b^{2}}{2b\sqrt {1-u^{2} }}\right]  \\
    R=&\frac{a^2-c^2u^2} {2b\sqrt{1-u^2}} \\
    C'=&\left[0,  - \,{\frac {2 b c^2\sqrt {1-u^{2}} }{a^{2}+  ( {u}^{2}-2
  ) c^{2}}} \right]
   \\
    R'=&\frac{c(c^2u^2-a^2)}{a^2+(u^2-2)c^2}
\end{align*}
\end{theorem}

\begin{corollary}
The half harmonic mean of $R^2$ and $R'^2$ is invariant and equal to $c^2=a^2-b^2$, i.e., $1/R^2+1/R'^2=1/c^{2}$.
\end{corollary}

Note: the above Pythagorean relation implies that the polygon whose vertices are a focus, and the inversion of $C,C',O$ (center of billiard) with respect to a unit circle centered on said focus, is a rectangle of sides $1/R$ and $1/R'$ and diagonal $1/c$.

Remarkably:

\begin{corollary}
Over the self-intersected $N=4$ family, the power of the origin with respect to both $\mathcal{C}$ and $\mathcal{C}'$ is invariant and equal to $b^2-a^2$. 
\end{corollary}

Referring to Figure~\ref{fig:4si-midpoints}: 

\begin{observation}
The midpoints of each of the four segments of self-intersected 4-periodics are collinear on a horizontal line. Their locus is an $\infty$-shaped quartic curve given by:

\[c^2(b^2x^2+a^2y^2)^2 - b^4a^2\left((a^2-2b^2)x^2 - a^2 y^2\right)=0.\]
\end{observation}

\noindent Furthermore, the above quartic is tangent to the confocal hyperbolic caustic at its vertices $[\pm a\sqrt{a^2-2 b^2}/c,0]$.

\begin{figure}
    \centering
    \includegraphics[width=.5\textwidth]{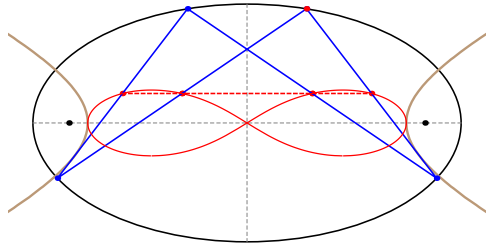}
    \caption{The midpoints of each of the four segments of self-intersected 4-periodics are collinear on a horizontal line. Their locus is an $\infty$-shaped quartic which touches the caustic at its vertices. \href{https://youtu.be/GZCrek7RTpQ}{Video}}
    \label{fig:4si-midpoints}
\end{figure}

Let $\mathcal{P}^\dagger$ (resp. $\mathcal{Q}^\dagger$) denote the inversive polygon of 4-periodics (resp. its outer polygon) wrt a unit circle $\mathcal{C}^\dagger$ centered on one focus. From properties of inversion:

\begin{corollary}
$\mathcal{P}^\dagger$  (resp. $\mathcal{Q}^\dagger$) has four collinear vertices, i.e., it degenerates to a segment along the radical axis of $\mathcal{C}^\dagger$ and $\mathcal{C}$ (resp. $\mathcal{C}'$).
\end{corollary}

\begin{proposition}
The two said radical axes are perpendicular.
\end{proposition}

\begin{proof}
 It is enough to check that the vectors $C-[-c,0]$ and $C'-[-c,0]$ are orthogonal. Observe that when $u^2=(a^2-2b^2)/c=(2c^2-a^2)/c$ the outer polygon is contained in the horizontal axis.
\end{proof}

\begin{observation}
The pairs of opposite sides of the outer polygon to self-intersected 4-periodics intersect at the top and bottom of the circle $(C,R)$ on which the 4-periodic vertices are concyclic. 
\end{observation}

\section{Deriving both Simple and Self-Intersected Invariants}
\label{sec:invariants}
In this section, we derive expressions for selected invariants introduced in \cite{reznik2020-invariants}, specifically for ``low-N'' cases, e.g., N=3,4,5,6,8. In that publication, each invariant is identified by a 3-digit code, e.g., $k_{101}$, $k_{102}$, etc. Table~\ref{tab:invariants} lists the invariants considered herein. The quantities involved are defined next.

Let $\theta_i$ denote the ith N-periodic angle.
Let $A$ the signed area of an N-periodic. Referring to Figure~\ref{fig:inner-outer}, singly-primed quantities (e.g., $\theta_i'$, $A'$, etc.), etc., always refer to the {\em outer polygon}: its  sides are tangent to the EB at the $P_i$. Likewise, doubly-primed quantities ($\theta_i''$, $A''$, etc.) refer to the {\em inner polygon}: its vertices lie at the touchpoints of N-periodic sides with the caustic. More details on said quantities appear in Appendix~\ref{app:invariants}.

Recall $k_{101}=J L-N$, as introduced in \cite{reznik2020-intelligencer,bialy2020-invariants}.

Referring to Figure~\ref{fig:inversive}, the $f_1$-{\em inversive polygon} has vertices at inversions of the $P_i$ with respect to a unit circle centered on $f_1$. Quantities such as $L_1^\dagger$, $A_1^\dagger$, etc., refer to perimeter, area, etc. of said polygon.

\begin{figure}
    \centering
    \includegraphics[width=.66\textwidth]{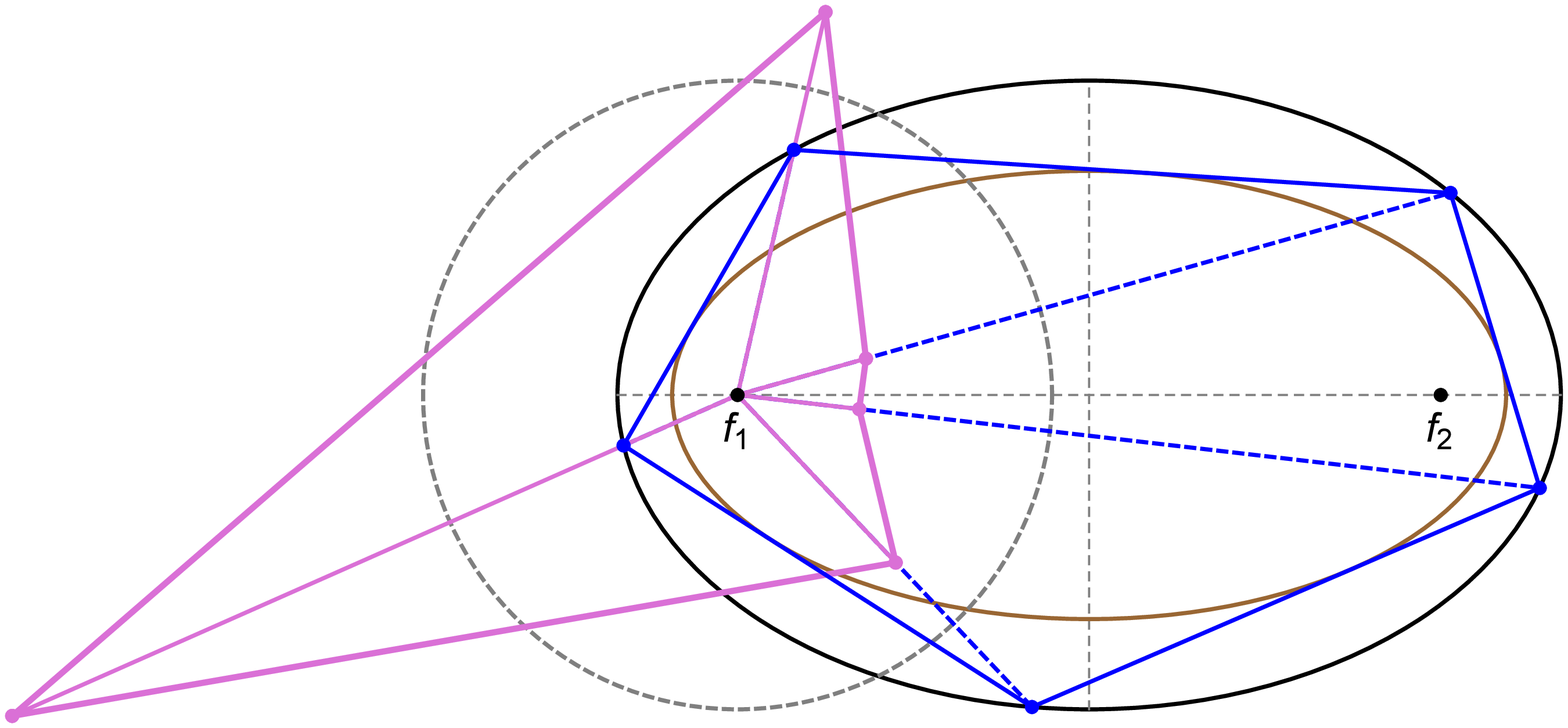}
      \caption{Focus-inversive 5-periodic (pink) whose vertices are inversions of the $P_i$ (blue) with respect to a unit circle (dashed black) centered on $f_1$. It turns out its perimeter is also invariant over the family as is the sum of its spoke lengths (pink lines) \href{https://youtu.be/wkstGKq5jOo}{Video}.}
    \label{fig:inversive}
\end{figure}

\begin{table}
\begin{tabular}{|c|l|c|l|l|c|}
\hline
code & invariant & valid N & derived & proofs \\
\hline
$k_{101}$ & $\sum\cos\theta_i$ & all & n/a & \cite{akopyan2020-invariants,bialy2020-invariants} \\
$k_{102}$ & $\prod{\cos\theta'_i}$ &  all & 3,4,5,$5_i$,6,$6_i$,$6_{ii}$ & \cite{akopyan2020-invariants,bialy2020-invariants} \\
$k_{103}$ & $A'/A$ &  odd & 3,5,$5_i$ & \cite{akopyan2020-invariants,caliz2020-area-product} \\
$k_{104}$ & $\sum{\cos(2\theta'_i)}$ &  all & 3,4,5,$5_i$,6,$6_i$,8 & \cite{akopyan19_private_half_sines} \\ 
$k_{105}$ & $\prod{\sin(\theta_i/2)}$ &  odd & 3,5,$5_i$ &  \cite{akopyan19_private_half_sines} \\ 
$k_{106}$ & $A'A$ &  even & 4,6,$6_i$,$6_{ii}$ & \cite{caliz2020-area-product} \\
$k_{110}$ & $A\,A''$ &  even & 4,6,$6_i$,$6_{ii}$ & ? \\
$^{\dagger}{k_{119}}$ & $\sum{\kappa_i^{2/3}}$ & all & 3,4,6 & \cite{akopyan2020-invariants,stachel2020-private} \\
\hline
$k_{802,a}$ & $\sum{1/d_{1,i}}$ &   all & 3,4,6 & \cite{akopyan2020-invariants} \\
$k_{803}$ & $L_1^\dagger$ &   all & 3,4,6 & ? \\
${^\dagger}{k_{804}}$ & $\sum\cos{\theta_{1,i}^\dagger}$ &   ${\neq}4$ & 3 & ? \\
$k_{805,a}$ & $A\,A_1^\dagger$ &  ${\equiv}\,0 \Mod{4}$ & 4,8 & ? \\
$k_{806}$ & $A/A_1^\dagger$ &  ${\equiv}\,2 \Mod{4}$ & 6 & ? \\
$k_{807}$ & $A_1^\dagger.A_2^\dagger$ &  odd & 3 & ? \\
\hline
\end{tabular}
\caption{List of selected invariants taken from \cite{reznik2020-invariants} as well as the low-N cases (column ``derived'') for expressions are derived herein. ${^\dagger}$ co-discovered with P. Roitman. A closed-form expression for $k_{119}$ was derived by H. Stachel; see \eqref{eqn:k119}.}.
\label{tab:invariants}
\end{table}

\subsection{Invariants for N=3}
\label{sec:invariants-n3}
AS before, let  $\delta=\sqrt{a^4-a^2 b^2+b^4}$.  For $N=3$ explicit expressions for $J$ and $L$ have been derived \cite{garcia2020-new-properties}:

\begin{align}
J=&\frac{\sqrt{2\delta-a^2-b^2}}{c^2}\nonumber\\
L=&2(\delta+a^2+b^2)J\label{eqn:lj}
\end{align}

\noindent When $a=b$, $J=\sqrt{3}/2$ and when $a/b{\to}\infty$, $J{\to}0$. 

\begin{proposition}
For $N=3$, $k_{102}=(J L)/4-1$.
\label{prop:n3-k102}
\end{proposition}

\begin{proof}
We've shown $\sum_{i=1}^3{\cos{\theta_i}}=JL-3$ is invariant for the $N=3$ family \cite{garcia2020-new-properties}. For any triangle $\sum_{i=1}^3{\cos{\theta_i}}=1+r/R$ \cite{mw}, so it follows that $r/R=JL-4$ is also invariant. Let $r_h,R_h$ be the Orthic Triangle's Inradius and Circumradius. The relation $r_h/ R_h=4\prod_{i=1}^3{|\cos\theta_i|}$ is well-known \cite[Orthic Triangle]{mw}. Since a triangle is the Orthic of its Excentral Triangle, we can write $r/R=4\prod_{i=1}^3{\cos\theta'_i}$, where $\theta_i'$ are the Excentral angles which are always acute \cite{mw} (absolute value can be dropped), yielding the claim.
\end{proof}

\begin{proposition}
For $N=3$, $k_{103}=k_{109}=2/(k_{101}-1)=2/(J L - 4)$.
\end{proposition}

\begin{proof}
Given a triangle $A'$ (resp. $A''$) refers to the area of the Excentral (resp. Extouch) triangles. The ratios $A'/A$ and $A/A''$ are equal. Actually, $A'/A=A/A''=(s_1 s_2 s_3)/(r^2 L)$, where $s_i$ are the sides, $L$ the perimeter, and $r$ the Inradius \cite[Excentral,Extouch]{mw}. Also known is that $A'/A=2R/r$ \cite{johnson29}. Since $r/R=\sum_{i=1}^3{\cos\theta_i}-1=k_{101}-1$ \cite[Inradius]{mw}, the result follows.
\end{proof}

\begin{proposition}
For $N=3$, $k_{104}=-k_{101}$ and is given by:

\[
 k_{104}={\frac { \left( {a}^{2}+{b}^{2} \right)  \left(  {a}^{2}+{b}^{2}-2\,
\delta \right) }{ c^4  } }=3-J L
\]
\end{proposition}

\begin{proposition}
For $N=3$, $k_{105}=(J L)/4-1 = k_{102}$.
\end{proposition}

\begin{proof}
Let $r,R$ be a triangle's Inradius and Circumradius. The identity $r/R=4\prod_{i=1}^3{\sin(\theta_i/2)}$ holds for any triangle \cite[Inradius]{mw}, which with Proposition~\ref{prop:n3-k102} This completes the proof.
\end{proof}

\begin{proposition}
For $N=3$, $k_{119}$ is given by:
 
\[(k_{119})^3=\frac{2J^3 L}{(JL-4)^2}, \;\;\;
k_{119}=\frac{a^2+b^2+\delta}{(ab)^{\frac{4}{3}} }
\]
\end{proposition}

\begin{proof}
Use the expressions for $L,J$ in \eqref{eqn:lj}.
\end{proof}

\begin{proposition}
For $N=3$, $k_{802,a}$ is given by:

\[  k_{802,a}=\frac {{a}^{2}+{b}^{2}+\delta}{a{b}^{
2}}= {\frac {J\sqrt{2}\sqrt{ J\,L+\sqrt {9-2\,JL}-3 }}{ J L-4 }}\]

\end{proposition}

\begin{remark}
For $N=3$, $k_{803}$ is given by:

\[k_{803}=\rho \frac {\sqrt { \left( 8\,{a}^{4}+4\,{a}^{2}{b}^{2}+2\,{b}^{4}
 \right) \delta+8\,{a}^{6}+3\,{a}^{2}{b}^{4}+2\,{b}^{6}}}{{a}^{2}{b}^{
2}}\]
\end{remark}

\begin{proposition}
For $N=3$, $k_{804}$ is given by:

\[k_{804}=\frac{\delta (a^2+c^2-\delta)}{a^2c^2}\]
\end{proposition}

\begin{proposition}
For $N=3$, $k_{807}$ is given by:

\[ k_{807}= \frac{\rho^8}{8 a^8 b^2}  \left[\left( {a}^{4}+2\,{a}^{2}{b}^{2}+4\,{b}^{4} \right)\delta +  {a}^{6}+(3/2)\,{a}^{4}{b}^{2}+4\,{b}^{6} \right]
\]
\end{proposition}

\noindent $r$ is the radius of the inversion circle, included above for unit consistency. By default $r=1$.

\subsection{Invariants for N=4}
\label{sec:invariants-n4}
\begin{proposition}
For $N=4$, $k_{102}= 0$.
\label{prop:n4-102}
\end{proposition}

\begin{proof}
Non-self-intersecting 4-periodics are parallelograms \cite{connes07} whose outer polygon is a rectangle inscribed in Monge's Orthoptic Circle \cite{reznik2020-intelligencer}. This finishes the proof.
\end{proof}

\begin{proposition}
For $N=4$, $k_{104}= -4$.
\label{prop:n4-104}
\end{proposition}

\begin{proof}
As in Proposition~\ref{prop:n4-102}, outer polygon is a rectangle.
\end{proof}

\begin{proposition}
For $N=4$, $k_{106}= 8 a^2 b^2$.
\label{thm:n4}
\end{proposition}

\begin{proposition}
For $N=4$, $k_{110}$ is given by $\frac{2a^4b^4}{(a^2+b^2)^2}$
\label{prop:k110}
\end{proposition}


\noindent Let $\kappa_a=(a b)^{-2/3}$ is the affine curvature of the ellipse and $r_m=\sqrt{a^2+b^2}$ the radius of Monge's orthoptic circle \cite{mw}.

\begin{proposition}
For $N=4$, $k_{119}= \frac{2(a^2+b^2)}{(ab)^{\frac{4}{3}}} = 2 (\kappa_a r_{m})^2$
\end{proposition}

\begin{proposition}
For $N=4$, $k_{802,a}=\frac{2(a^2+b^2)}{ab^2}$
\end{proposition}

\begin{proposition}
For $N=4$, $k_{803}= \,{\frac {4\sqrt {{a}^{2}+{b}^{2}}}{{b}^{2}}}$
\end{proposition}

\begin{observation}
Experimentally, $k_{804}$ is invariant for all non-intersecting N-periodics, except for $N=4$.
\label{obs:n4-804}
\end{observation}

\begin{proposition}
For $N=4$, $k_{805,a}=4$
\end{proposition}

Note: see Section~\ref{sec:n4-si} for a treatment of N=4 self-intersected geometry.

\subsection{Invariants for N=5}
\label{sec:invariants-n5}
As seen in Appendix~\ref{app:vertices-caustics}, the vertices of 5-periodics can only be obtained via an implicitly-defined caustic. I.e., we first numerically obtain the caustic semi-axes and then compute a axis-symmetric polygon tangent to it. Note that both simple and self-intersected 5-periodics possess an elliptic confocal caustic; see Figure~\ref{fig:n5-both}.

\begin{figure}
    \centering
    \includegraphics[width=.8\textwidth]{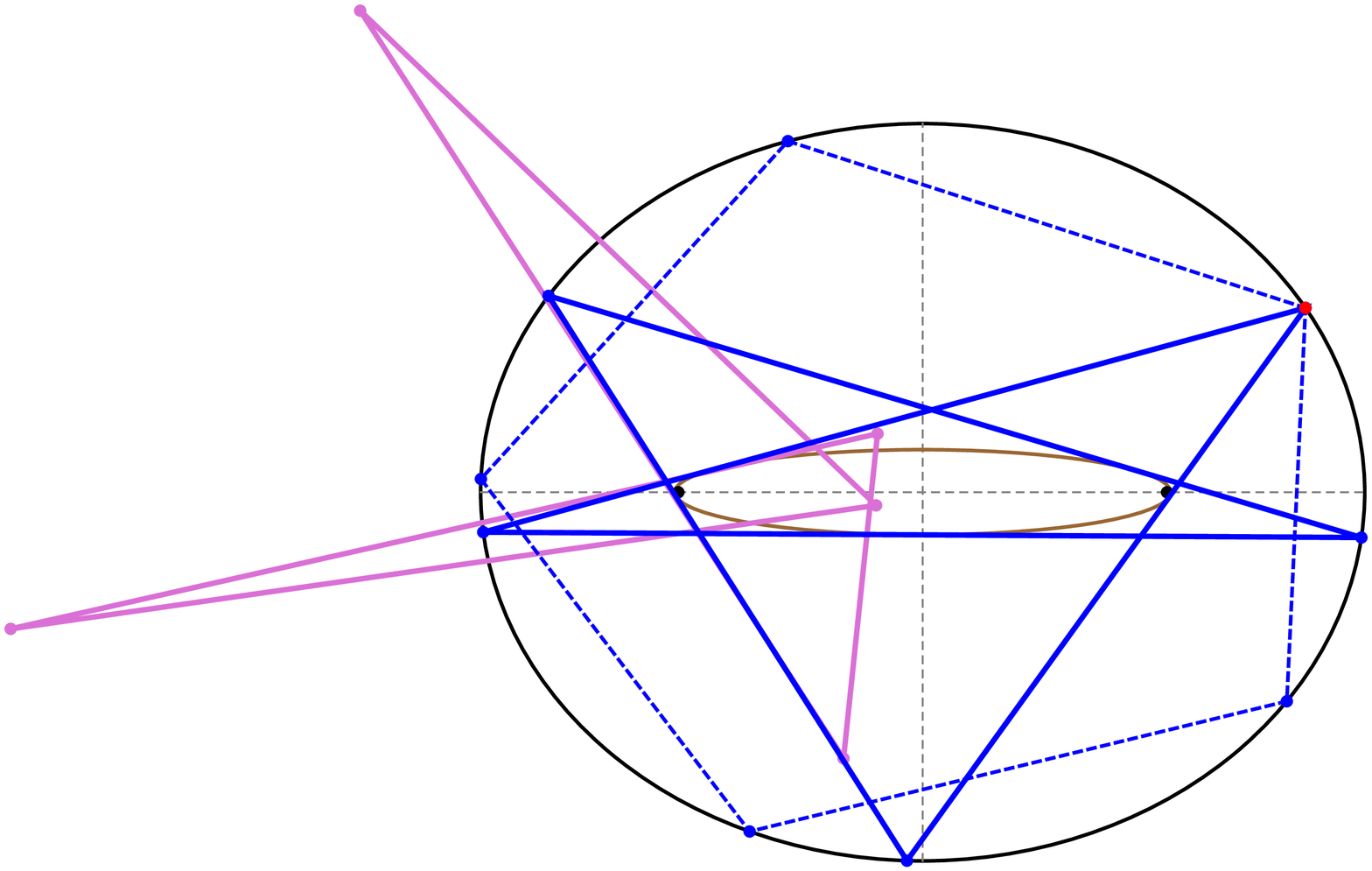}
    \caption{A simple (blue) and self-intersected (dashed blue) 5-periodic, as well as the former's focus-inversive polygon (pink). \href{https://youtu.be/LuLtbwkfSbc}{Video}}
    \label{fig:n5-both}
\end{figure}

\begin{proposition}
For $N=5$ simple (resp. self-intersected), $k_{102}$ is given by the largest negative (resp. positive) real root of the following sextic: 

\begin{align*}
k_{102}: &\;\; 1024   c^{20} x^6+2048   (a^4+a^3   b-a   b^3+b^4)  (a^4-a^3   b+a   b^3+b^4)   c^{12}x^5\\
&+256   (4   a^{12}-a^{10}   b^2+32   a^8   b^4-22   a^6   b^6+32   a^4   b^8-a^2   b^{10}+4   b^{12})c^8 x^4\\
&-64   a^2   b^2   (4   a^{12}-27   a^{10 }  b^2+38   a^8   b^4-126   a^6   b^6+38   a^4   b^8-27   a^2   b^{10}+4   b^{12})c^4   x^3\\
&-16   a^6   b^6   (7   a^8-96   a^6   b^2+114   a^4   b^4-96   a^2   b^6+7   b^8)   x^2\\
&-8   a^8   b^8   (7   a^4+30   a^2   b^2+7   b^4)   x-a^{10}   b^{10}=0
\end{align*}
\end{proposition}


\begin{proposition}
For $N=5$ simple (resp. self-intersected), $k_{103}$ is given by the smallest (resp. largest) real root greater than 1 of the following sextic: 

\begin{align*}
k_{103}: &\;
 {a}^{6}{b}^{6}x^6-2\,{b}^{2}{a}^{2} \left( 4\,{a}^{8}-{a}^{6}{b}^{
2}-{a}^{2}{b}^{6}+4\,{b}^{8} \right) x^{5}\\
&-{b}^{2}{a}^{2} \left( 4\,
{a}^{8}+19\,{a}^{6}{b}^{2}-62\,{a}^{4}{b}^{4}+19\,{a}^{2}{b}^{6}+4\,{b
}^{8} \right) x^{4}\\
&+12\,{b}^{2}{a}^{2} \left( {a}^{4}+{b}^{4}
 \right) {c}^{4}x^{3}+ \left( 4\,{a}^{8}+19\,{a}^{6}{b}^{2}+66\,{a}^
{4}{b}^{4}+19\,{a}^{2}{b}^{6}+4\,{b}^{8} \right) {c}^{4}x^{2}\\
&+
 \left( 2\,{a}^{8}+12\,{a}^{6}{b}^{2}+36\,{a}^{4}{b}^{4}+12\,{a}^{2}{b
}^{6}+2\,{b}^{8} \right) {c}^{4}x-{c}^{12}
\end{align*}
\end{proposition}

\begin{proposition}
For $N=5$ simple (resp. self-intersected), $k_{104}$ is given by the only negative (resp. smallest largest) real root of the following sextic:

\begin{align*}
 & \; c^{12}
   x^6 - 2(a^4 + 10a^2b^2 + b^4) c^8
   x^5 - (37a^4 - 6a^2b^2 + 37b^4) c^8
   x^4 \\
  & +  4(5a^8 + 92a^6b^2 + 62a^4b^4 + 92a^2b^6 +
       5b^8) c^4
   x^3\\
   &+ (423a^{12} - 354a^{10}b^2 + 2713a^8b^4 - 4796a^6b^6 +
     2713a^4b^8 - 354a^2b^{10} + 423b^{12})
   x^2 \\
   &+ (270a^{12} + 740a^{10}b^2 - 3630a^8b^4 + 7160a^6b^6 -
     3630a^4b^8 + 740a^2b^{10} + 270b^{12})x\\
     &- 675a^{12} -
  850a^{10}b^2 + 1075a^8b^4 - 3900a^6b^6 + 1075a^4b^8 -  850a^2b^{10} - 675b^{12} =  0
  \end{align*}
\end{proposition}

\begin{proposition}
For $N=5$ simple (resp. self-intersected), $k_{105}$ is given by the largest positive real root (resp. the symmetric value of the largest negative root) of the following sextic: 
{\small 

\[\arraycolsep=0pt\begin{array}{rcl}
 2^{10}c^{20} x^6 &&+ 2^{10} (2 a^{12} + a^{10} b^2 + 26 a^8 b^4 + 70 a^6 b^6 + 26 a^4 b^8 + a^2 b^{10 }+ 2 b^{12}) c^8 x^5 \\
   + 2^{8} &(&4 a^{12} + 30 a^{10} b^2 + 71 a^8 b^4 + 350 a^6 b^6 + 
       71 a^4 b^8 + 30 a^2 b^{10} + 4 b^{12}) c^8 x^4\\
       + 2^{6} a^2 b^2 &(&4 a^{12} + 9 a^{10} b^2 - 318 a^8 b^4 - 126 a^6 b^6 - 
     318 a^4 b^8 + 9 a^2 b^{10} + 4 b^{12}) c^4x^3\\
     -
  2^6 a^2 b^2 &(&8 a^{16} - 53 a^{14} b^2 + 253 a^{12} b^4 - 1041 a^{10} b^6 + 
     1650 a^8 b^8 \\
      &&- 1041 a^6 b^{10} + 253 a^4 b^{12} - 53 a^2 b^{14} + 
     8 b^{16}) x^2 \\
     - 
  2^{4} a^2 b^2 &(&16 a^{16} - 12 a^{14} b^2 + 5 a^{12} b^4 + a^{10}b^6 + 
     2a^8b^8 + a^6b^{10} + 5 a^4 b^{12} - 12 a^2 b^{14} + 16 b^{16}) x\\
    - 
a^{10} b^{10}&&=0
 \end{array}\]
}

\end{proposition}

\subsection{Invariants for N=6}
\label{sec:invariants-n6}
\noindent Referring to Figure~\ref{fig:inner-outer} (right):
\begin{proposition}
For $N=6$, $k_{102}=a^2b^2/(4(a+b)^4) = (JL-4)^2/64$.
\label{thm:n6k102}
\end{proposition}

\begin{proposition}
For $N=6$, $k_{104}=k_{101}=J L-6$.
\label{thm:n6k104}
\end{proposition}

\begin{proposition}
For $N=6$, $k_{106}$ is given by:

\[ k_{106}=  \, \,{\frac {4{b}^{2} \left( 2\,a+b \right) {a}^{2} \left( a+2\,b
 \right) }{ \left( a+b \right) ^{2}}}
  =  -\,{\frac { \left( J\,L-12 \right)  \left( J
\,L-4 \right) ^{2}}{16 J^{4}}}
\]
\label{thm:n6k106}
\end{proposition}

\begin{proposition}
For $N=6$, $k_{110}$ is given by: 

\[ k_{110}=    {\frac { 4 {a}^{3}{b}^{3} \left( 2\,a+b \right) ^{2} \left( a+2\,b
 \right)^2 }{ \left( a+b \right) ^{6}}}
  =   -{\frac { \left( J\,L-12 \right) ^{2} \left( J\,L-4 \right) ^{3}}{256\,J^{4}}}
\]
\label{thm:n6k110}
\end{proposition}


\begin{proposition}
For $N=6$, $k_{119}^3= \frac{2^5 J^5 L^3}{(JL-4)^4}$ 
\end{proposition}

\begin{proposition}
For $N=6$, $k_{802,a}=\frac{2(a^2+ab+b^2)}{ab^2}=\frac{ 4 J^2L\left(1+\sqrt{J L-3}\right) }{(J L-4)^2}$ 
\end{proposition}

\begin{proposition}
For $N=6$, $k_{803}=
2{\rho^2}(2a^2+2ab-b^2)/(ab^2)$. 
\label{prop:k803}
\end{proposition}

\begin{proposition}
For $N=6$, $k_{806}=
4{\rho^{-4}}a^3b^4/((2a-b)(a+b)^2)$. 
\label{prop:k806}
\end{proposition}

Referring to Figure~\ref{fig:n6-I}:

\subsection{N=6 self-intersecting}

Take a regular hexagon. There are 60
orderings with which to connect its vertices (modulo rotations and chirality); see \cite[Pascal Lines]{mw}.

It turns out only two $N=6$ topologies can produce closed trajectories, both with two self-intersections. These will be referred to as type I and type II, and are depicted in Figures \ref{fig:n6-I}, and \ref{fig:n6-II}, respectively.

\begin{figure}
    \centering
    \includegraphics[width=\textwidth]{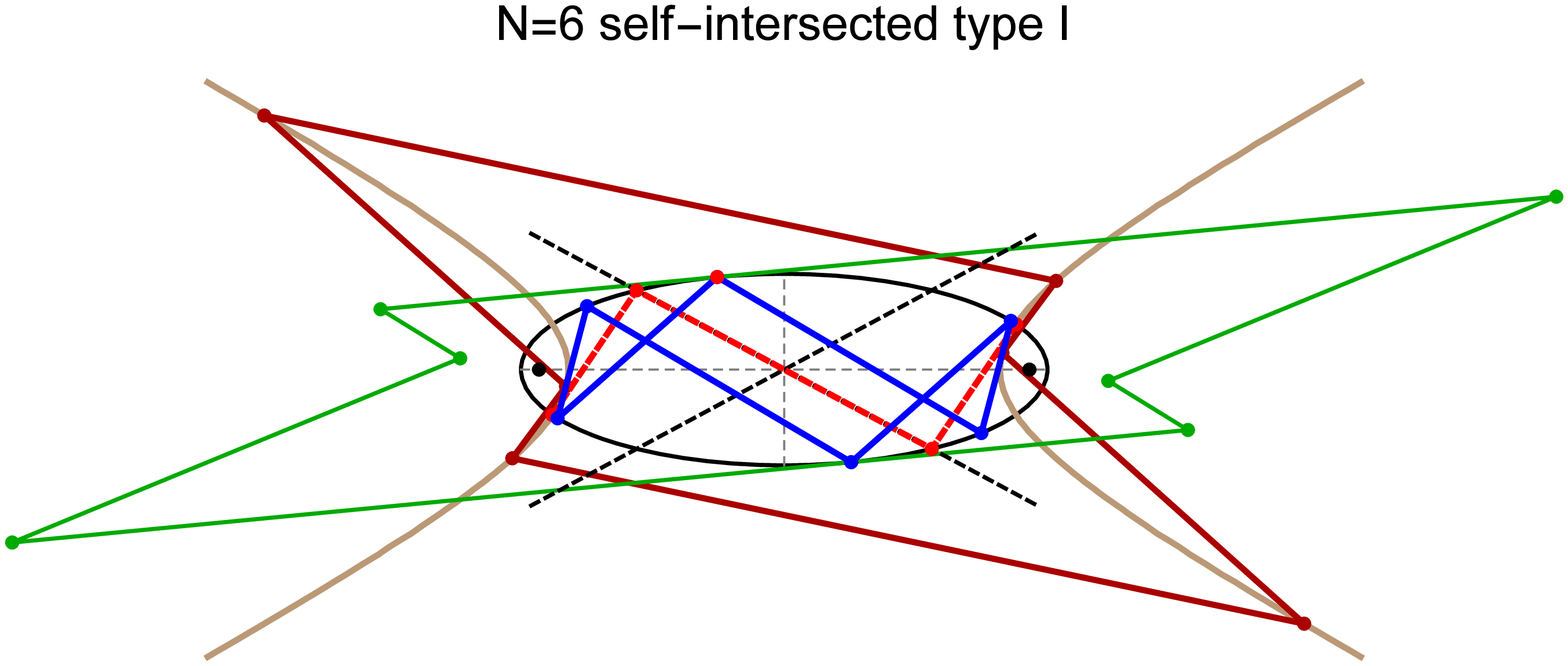}
    \caption{Type I self-intersected 6-periodic (blue) and its doubled-up configuration (dashed red), both tangent to a hyperbolic confocal caustic (brown). Its asymptotes (dashed black) pass through the center of the EB. Also shown are the outer (green) and inner (dark red) polygons. \href{https://youtu.be/fOD85MNrmdQ}{Video}}
    \label{fig:n6-I}
\end{figure}

\begin{figure}
    \centering
    \includegraphics[width=1\textwidth]{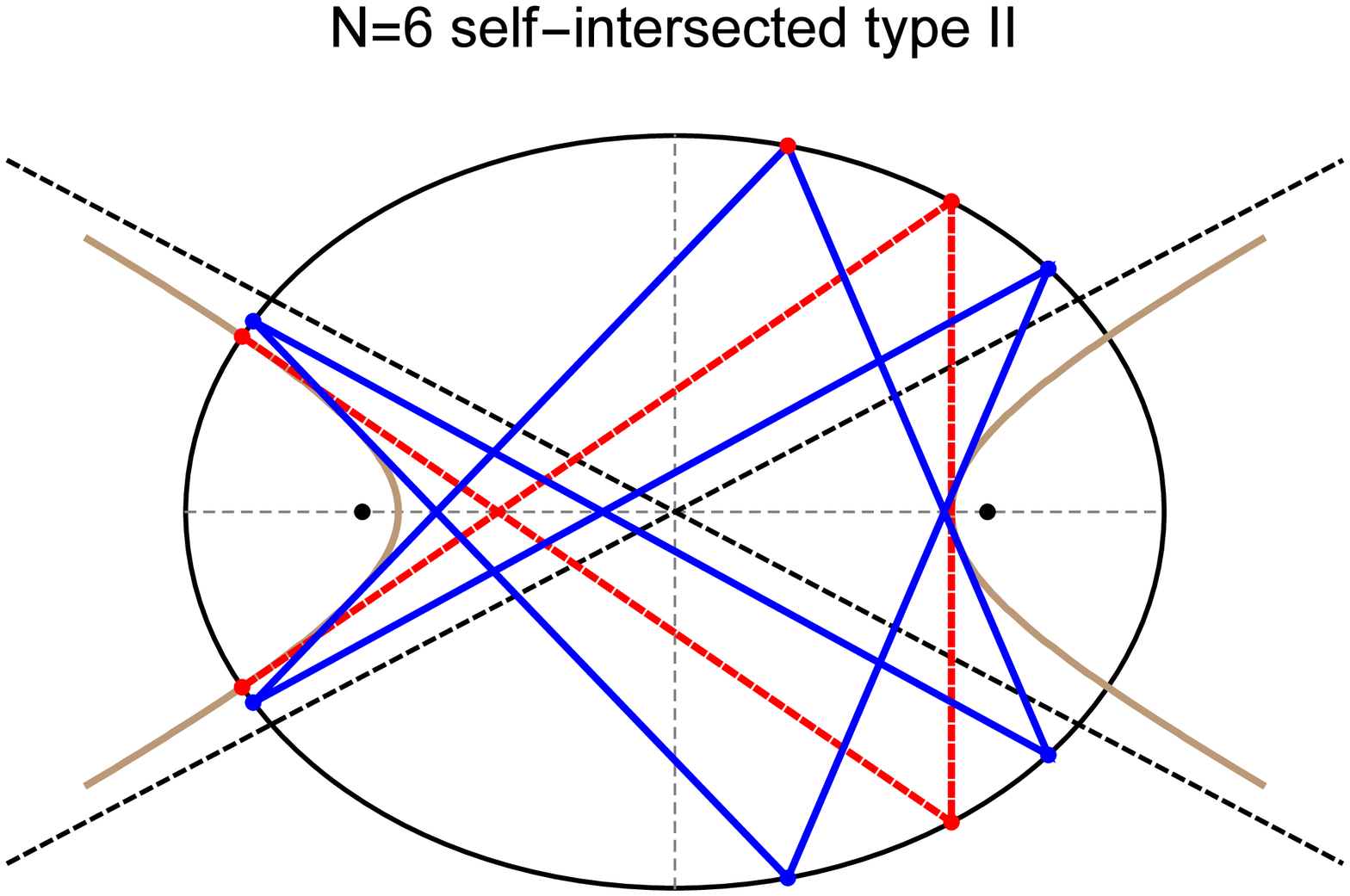}
    \caption{Self-intersected 6-periodic (type II) shown both at one of its doubled-up configurations (dashed red) and in general position (blue). Segments are tangnet to a hyperbolic confocal caustic (brown) whose asymptotes (dashed black) pass through the center of the EB. Also shown is the outer polygon (green) which in this case is always simple. \href{https://youtu.be/gQ-FbSq7wWY}{Video}}
    \label{fig:n6-II}
\end{figure}

\begin{proposition} 
For $N=6$ type I,  $k_{102}= a^2b^2/(4(a-b)^4) = (J\;L-4)^2/64$ 
\label{thm:n6k102I}
\end{proposition}

\begin{proposition} 
For $N=6$ type II,
\[k_{102}= a^2(a-c)^2/(4c^4)= (JL-8)^2(JL-4)^2/1024\] 
\label{thm:n6k102II}
\end{proposition}

\begin{proposition} 
For $N=6$ type I,  $k_{104}= - \,{\frac {2({a}^{2}-4\,a b+{b}^{2})}{ \left( a-b \right) ^{2}}} =JL-6 = k_{101}$.
\label{thm:n6k104I}
\end{proposition}

Note that $k_{104}=k_{101}=J L-6$ for both $N=6$ simple and type I. However:

\begin{proposition} 
For $N=6$ type II,  \[k_{104}= \frac{2(a^2-ac+c^2)(a^2-ac-c^2)}{c^4} = \frac{ (J^2L^2-12JL+16)(J^2L^2-12JL+48)}{128} \]
\label{thm:n6k104II}
\end{proposition}
 
\begin{proposition} 
For $N=6$ type I,  $k_{106}=  4 a^2b^2(a-2b)(2a-b)/(a-b)^2 = -(JL-12)(JL-4)^2/(16 J^4)$ 
\label{thm:n6k106I}
\end{proposition}
 
\begin{proposition} 
For $N=6$ type I,  
\[k_{110}=   {\frac {-4{a}^{3}{b}^{3} \left( a-2\,b \right) ^{2} \left( 2\,a-b
 \right)^2  }{ \left( a-b \right) ^{6}}}
 =\frac { \left( JL-12 \right)^2  \left( JL-4 \right) ^{3}  }{2^8\,J^{4}}
\]
\label{thm:n6k110I}
\end{proposition}

\begin{proposition}
For $N=6$ type II, both $A$ and $A'$ vanish, and therefore $k_{106}=0$ and $k_{110}=0$.
\end{proposition}

\begin{observation}
Experimentally, $k_{804}$ is invariant for $N=6$ simple, and type I. However, it is variable for $N=6$ type II.
\label{obs:n6ii-804}
\end{observation}

\subsection{Invariants for N=8}
\label{sec:invariants-n8}
Referring to Figure~\ref{fig:n8-c2t}:

\begin{figure}
    \centering
    \includegraphics[width=.66\textwidth]{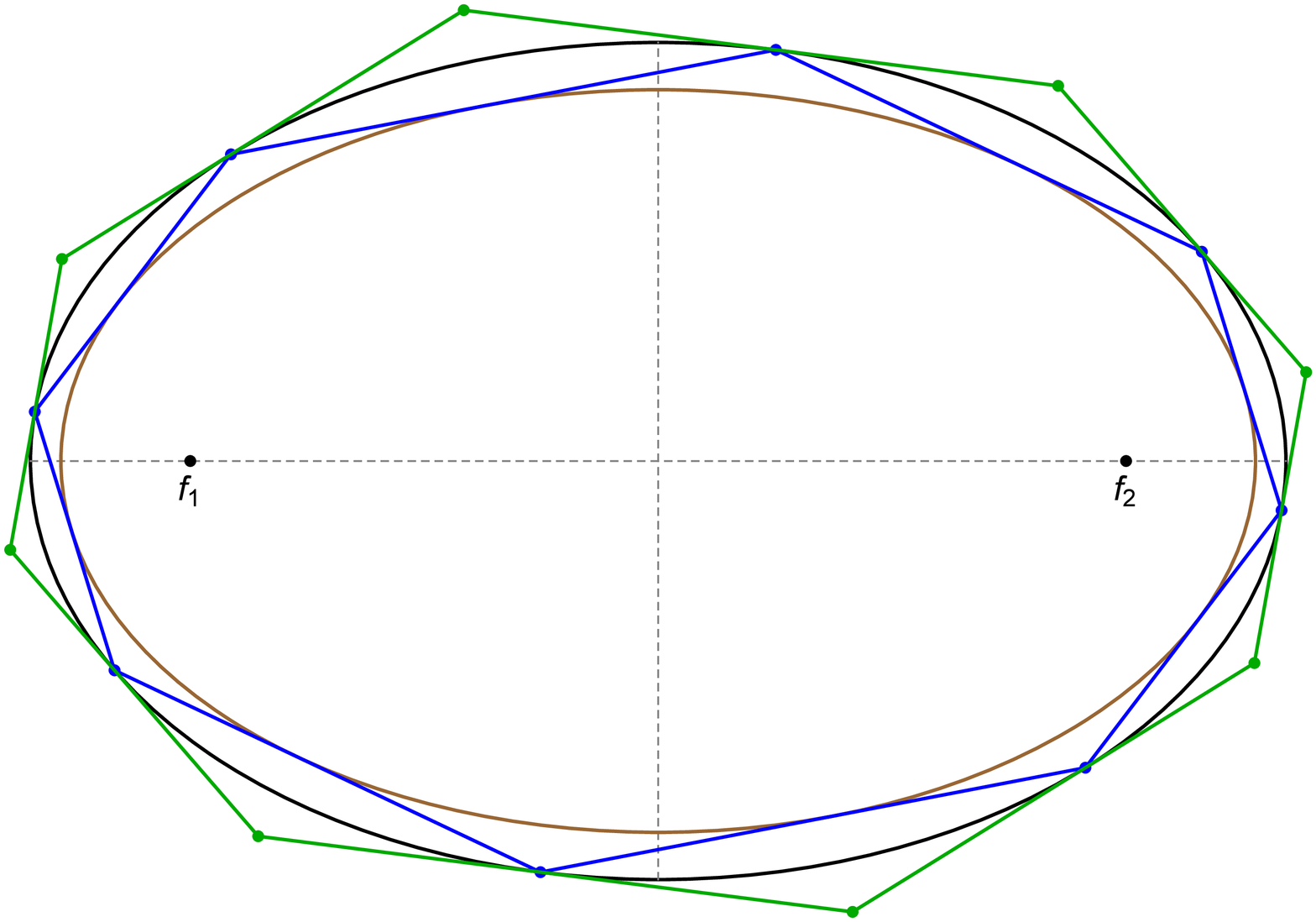}
    \caption{The outer polygon (green) to an 8-periodic has null sum of double cosines. \href{https://youtu.be/GEmV_U4eRIE}{Video}}
    \label{fig:n8-c2t}
\end{figure}

\begin{proposition}
For $N=8$, $k_{102} $ is given by $(1/2^{12})(J L-4)^2 (J L-12)^2$.
\label{thm:n8-102}
\end{proposition}

\noindent Referring to Figure~\ref{fig:n8-c2t}:

\begin{proposition}
For $N=8$, $k_{104}=0$.
\label{thm:n8-104}
\end{proposition}

\begin{proof}
Using the CAS, we checked that $k_{104}$ vanishes for an 8-periodic in the ``horizontal'' position, i.e., $P_1=(a,0)$. Since $k_{104}$ is invariant \cite{akopyan2020-invariants}, this completes the proof.
\end{proof}

\subsection{N=8 Self-Intersected}

There are 3 types of self-intersected 8-periodics \cite{birkhoff1927}, here called type I, II, and III. These correspond to trajectories with turning numbers of 0, 2, and 3, respectively. These are depicted in Figures~\ref{fig:n8-si-I}, \ref{fig:n8-si-II}, and \ref{fig:n8-si-III}.

\begin{observation}
The signed area of $N=8$ type I is zero.
\end{observation}

Referring to Figure~\ref{fig:n8-si-III}, the following is related to the Poncelet Grid \cite{sergei07_grid} and the Hexagramma Mysticum \cite{baralic2015}:

\begin{observation}
The outer polygon to N=8 type III is inscribed in an ellipse. 
\end{observation}

\begin{figure}
    \centering
    \includegraphics[width=\textwidth]{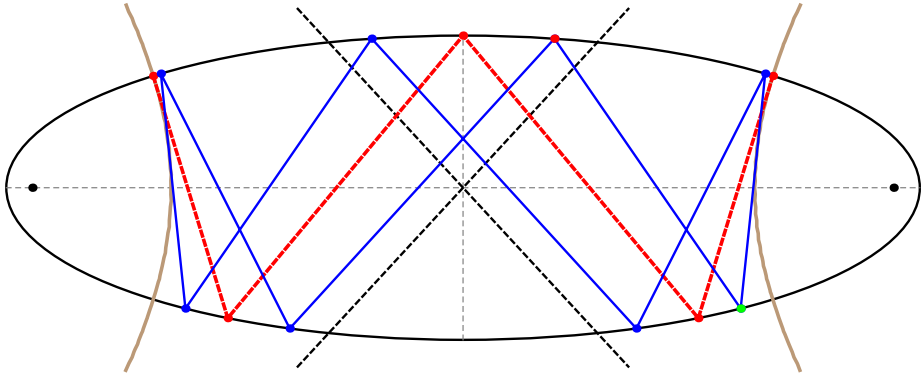}
    \caption{Self-intersecting 8-periodic of type I (blue) and its doubled-up configuration (dashed red) in $a/b=3$ ellipse. Trajectory segments are tangent to a confocal hyperbolic caustic (brown). \href{https://youtu.be/5Lt9atsZhRs}{Video}}
    \label{fig:n8-si-I}
\end{figure}

\begin{figure}
    \centering
    \includegraphics[width=\textwidth]{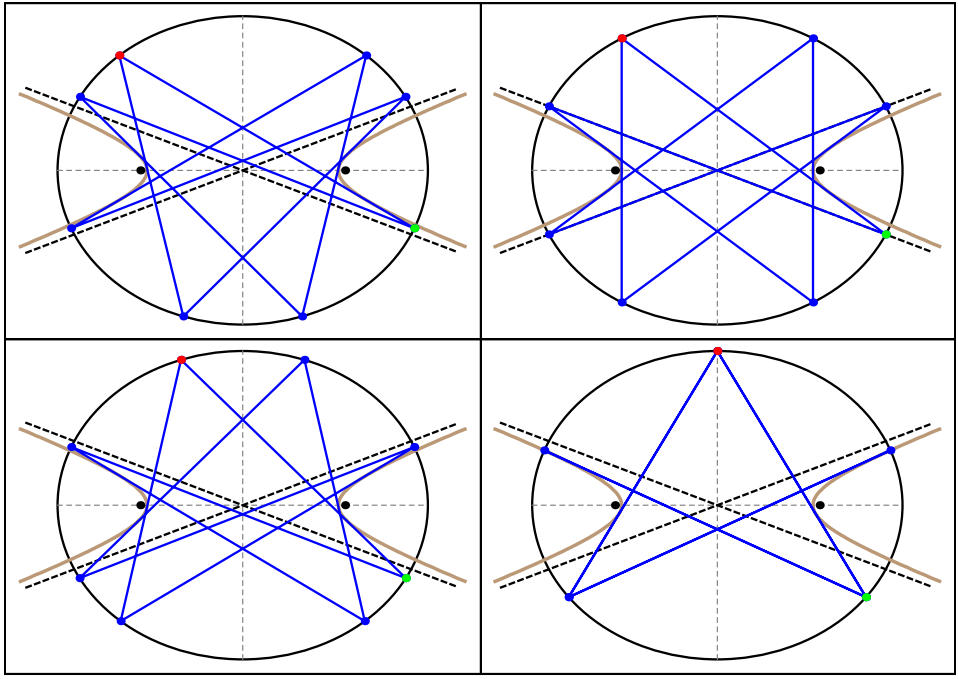}
    \caption{Four positions of a type II self-intersecting 8-periodic (blue) in an $a/b=1.2$ elliptic billiard, at four different locations of a starting vertex (red). In general position, these have turning number 2. Also shown is confocal hyperbolic caustic (brown). \href{https://youtu.be/JwD_w5ecPYs}{Video}}
    \label{fig:n8-si-II}
\end{figure}

\begin{figure}
    \centering
    \includegraphics[width=\textwidth]{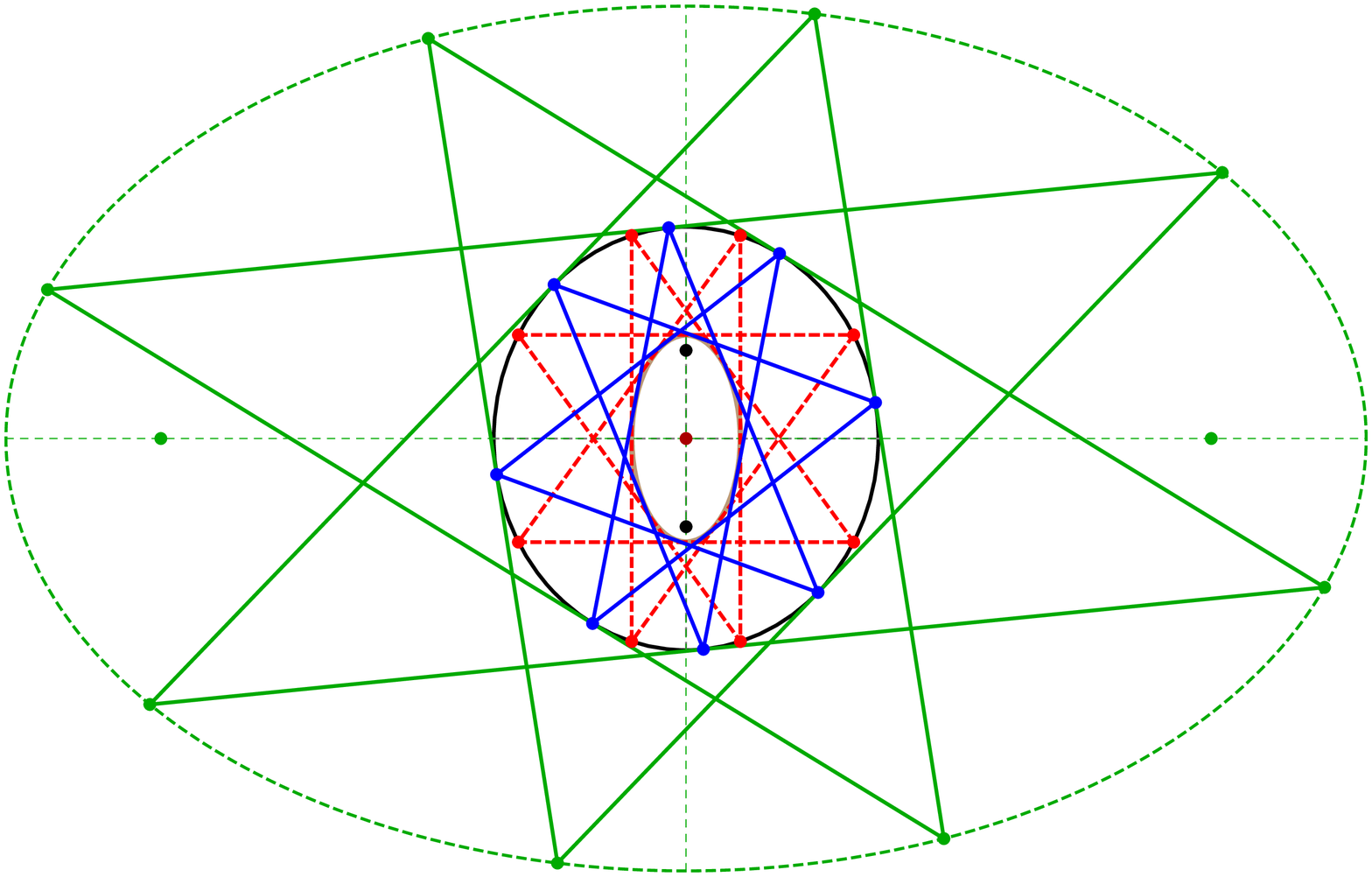}
    \caption{Self-intersecting 8-periodic of type III (blue) and its doubled-up configuration (dashed red) in an $a/b=1.1$ billiard ellipse (shown rotated by $90^\circ$ to save space). The turning number is 3. The confocal caustic is an ellipse (brown). Also shown is the outer polygon (green) whose vertices are inscribed in an axis-aligned, concentric ellipse (dashed green), a result related to \cite{baralic2015}. \href{https://youtu.be/93xpGnDxyi0}{Video}}
    \label{fig:n8-si-III}
\end{figure}

\section{Conclusion}
\label{sec:conclusion}
The sum of focus-inversive cosines ($k_{804}$) is  invariant in all N-periodics so far studied, excepting $N=4$ simple and $N=6$ type II; see Observations~\ref{obs:n4-804} and \ref{obs:n6ii-804}. Notice that for the former
the area of simple 4-periodics is non-zero while the sum of cosines vanishes; in the latter case the situation reverses: the area of type II 6-periodics vanishes and the sum of cosines is non-zero. Could either vanishing quantity be the reason $k_{804}$ becomes variable, e.g., is this  introducing a pole in the meromorphic functions on the elliptic curve \cite{akopyan2020-invariants}?

\section{Videos}
\label{sec:videos}
Animations illustrating some of the above phenomena are listed on Table~\ref{tab:playlist}.

\begin{table}
\small
\begin{tabular}{|c|l|l|}
\hline
id & Title & \textbf{youtu.be/<...>}\\
\hline
01 & {N=3--6 orbits and caustics} & 
\href{https://youtu.be/Y3q35DObfZU}{\texttt{Y3q35DObfZU}}\\
02 & {N=5 with inner and outer polygons} & \href{https://youtu.be/PRkhrUNTXd8}{\texttt{PRkhrUNTXd8}} \\
03 & {N=6 zero-area antipedal @ $a/b=2$} &
\href{https://youtu.be/HMhZW_kWLGw}{\texttt{HMhZW\_kWLGw}} \\
04 & {N=8 outer poly's null sum of double cosines} &
\href{https://youtu.be/GEmV_U4eRIE}{\texttt{GEmV\_U4eRIE}}\\
\hline
05 & {N=4 self-int. and its outer polygon} & \href{https://youtu.be/C8W2e6ftfOw}{\texttt{C8W2e6ftfOw}} \\
06 & {N=4 self-int. vertices concyclic w/ foci} &
\href{https://youtu.be/207Ta31Pl9I}{\texttt{207Ta31Pl9I}} \\
07 & {N=4 self-int. vertices and outer concyclic w/ foci} &
\href{https://youtu.be/4g-JBshX10U}{\texttt{4g-JBshX10U}} \\
08 & {N=4 self-int. collinear segment midpoints and 8-shaped locus} &
\href{https://youtu.be/GZCrek7RTpQ}{\texttt{GZCrek7RTpQ}} \\
09 & {N=5 self-int. (pentagram)} & 
\href{https://youtu.be/ECe4DptduJY}{\texttt{ECe4DptduJY}} \\
10 & {N=6 self-int. type I} & 
\href{https://youtu.be/fOD85MNrmdQ}{\texttt{fOD85MNrmdQ}} \\
11 & {N=6 self-int. type II} & 
\href{https://youtu.be/gQ-FbSq7wWY}{\texttt{gQ-FbSq7wWY}}\\
12 & {N=7 self-int. type I and II} &
\href{https://youtu.be/yzBG8rgPUP4}{\texttt{yzBG8rgPUP4}}\\
13 & {N=8 self-int. type I} &
\href{https://youtu.be/5Lt9atsZhRs}{\texttt{5Lt9atsZhRs}} \\
14 & {N=8 self-int. type II} &
\href{https://youtu.be/93xpGnDxyi0}{\texttt{3xpGnDxyi0}} \\
15 & {N=8 self-int. type III} &
\href{https://youtu.be/JwD_w5ecPYs}{\texttt{JwD\_w5ecPYs}} \\
\hline
16 & {N=3 inversives rigidly-moving circumbilliard} & 
\href{https://youtu.be/LOJK5izTctI}{\texttt{LOJK5izTctI}}\\
17 & {N=3 inversive: invariant area product} & 
\href{https://youtu.be/0L2uMk2xyKk}{\texttt{0L2uMk2xyKk}}\\
18 & {N=5 inversives: invariant area product} &
\href{https://youtu.be/bTkbdEPNUOY}{bTkbdEPNUOY} \\
19 & {N=5 self-int. inversive: invariant perimenter} &
\href{https://youtu.be/LuLtbwkfSbc}{\texttt{LuLtbwkfSbc}}\\
20 & {N=5 and outer inversives: invariant area ratio} &
\href{https://youtu.be/eG4UCgMkKl8}{\texttt{eG4UCgMkKl8}} \\
21 & {N=7 self-int. type I inversives: invariant area product} & \href{https://youtu.be/BRQ39O9ogNE}{\texttt{BRQ39O9ogNE}} \\
\hline
\end{tabular}
\caption{Videos illustrating some concepts in the article. The last column is clickable and/or provides the YouTube code.}
\label{tab:playlist}
\end{table}

\section*{Acknowledgments}
We would like to thank Jair Koiller, Arseniy Akopyan, Richard Schwartz, Sergei Tabachnikov, Pedro Roitman, and Hellmuth Stachel for useful insights.

The first author is fellow of CNPq and coordinator of Project PRONEX/ CNPq/ FAPEG 2017 10 26 7000 508.

\appendix

\section{Review: Elliptic Billiard}
\label{app:invariants}
Joachimsthal's Integral expresses that every trajectory segment is tangent to a confocal caustic \cite{sergei91}. Equivalently, a positive quantity $J$ remains invariant at every bounce point $P_i=(x_i,y_i)$:

\begin{equation*}
 J=\frac{1}{2}\nabla{f_i}.\hat{v}=\frac{1}{2}|\nabla{f_i}|\cos\alpha
\end{equation*}

\noindent where $\hat{v}$ is the unit incoming (or outgoing) velocity vector, and:

\begin{equation*}
\nabla{f_i}=2\left(\frac{x_i}{a^2}\,,\frac{y_i}{b^2}\right).
\label{eqn:fnable}
\end{equation*}

Hellmuth Stachel contributed \cite{stachel2020-private} an elegant expression for Joahmisthal's constant $J$ in terms of EB semiaxes $a,b$ and the major semiaxes $a''$ of the caustic:

\begin{equation*}
    J = \frac{\sqrt{a^2 - a''^2}}{{a}{b}}
\end{equation*}

Let $\kappa_i$ denote the curvature of the EB at $P_i$ given by \cite[Ellipse]{mw}:

\begin{equation}
\kappa = \frac{1}{a^2 b^2} \left(\frac{x^2}{a^4}+\frac{y^2}{b^4}\right)^{-3/2}
\label{eqn:curv}
\end{equation}

\noindent The signed area of a polygon is given by the following sum of cross-products \cite{preparata1988}:

\[  A=\frac{1}{2}\sum_{i=1}^N
{(P_{i+1}-P_{i})\times(P_i-P_{i+1})} \]

Let $d_{j,i}$ be the distance $|P_i-f_j|$. The inversion $P_{j,i}^\dagger$ of vertex $P_i$ with respect to a circle of radius $\rho$ centered on $f_j$ is given by:

\[ P_{j,i}^\dagger=f_j+ \left(\frac{\rho}{d_{j,i}}\right)^2 (P_i-f_j)\]

The following closed-form expression for $k_{119}$ for all $N$ was contributed by H. Stachel \cite{stachel2020-private}:

\begin{equation}
\sum_{i=1}^N{\kappa_i^{2/3}} = L/[2 J (a b)^{4/3}]
\label{eqn:k119}
\end{equation}

\section{Vertices \& Caustics N=3,4,5,6,8}
\label{app:vertices-caustics}
The four intersections of an ellipse with semi-axes $a,b$ with a confocal hyperbola with semi-axes $a'',b''$ are given by:

\subsection{N=3 Vertices \& Caustic}
\label{app:3-periodic}
Let $P_i=(x_i,y_i)/q_i$, $i=1,2,3$, denote the 3-periodic vertices, given by \cite{garcia2019-incenter}:

\begin{align*}
q_1&=1\\
x_{2}&=-{b}^{4} \left(  \left(   a^2+{b}^{2}\right)k_1 -{a}^{2}  \right) x_1^{3}-2\,{a}^{4}{b}^{2} k_2  x_1^{2}{y_1}\\
&+{a}^{4} \left(  ({a
}^{2}-3\, {b}^{2})k_1+{b}^{2}
 \right) {x_1}\,y_1^{2}-2{a}^{6} k_2 y_1^{3}\\
y_{2}&= 2{b}^{6} k_2 x_1^{3}+{b}^{4}\left(  ({b
 }^{2}-3\, {a}^{2}) k_1  +{a}^{2}
  \right) x_1^{2}{y_1}\\
&+  2\,{a}^{2} {b}^{4}k_2 {x_1} y_1^{2} -{
a}^{4}  \left(  \left(   a^2+{b}^{2}\right)k_1  -{b}^{2}  \right)  y_1^{3}
\\
q_2&={b}^{4} \left( a^2-c^2k_1   \right)
x_1^{2}+{a}^{4} \left(  {b}^{2}+c^2k_1  
 \right) y_1^{2} - 2\, {a}^{2}{b}^{2}{c^2}k_2 {x_1}\,{
y_1} \\
x_{3}&= {b}^{4} \left( {a}^{2}- \left( {b}^{2}+{a}^{2} \right) \right)
 k_1  x_1^{3} +2\,{a}^{4}{b}^{2}k_2  x_1^{2}{ y_1}\\
 &+{a}^{4} \left( 
  k_1 \left( {a}^{2}-3\,{b}^{2}
 \right) +{b}^{2} \right) { x_1}\, y_1^{2} +2\, {a}^{6} k_2 y_1^{3}
\\
y_{3}&= -2\, {b}^{6} k_2 x_1^{3}+{b}^{4} \left( {a}^{2}+ \left( {b}^{2}-3\,{a}^{2} \right)    k_1 \right) {{ x_1}}^{2}{ y_1}
\\
& -2\,{a}^{2}  {b}^{4} k_2  x_1 y_1^{2}+
 {a}^{4} \left( {b}^{2}- \left( {b}^{2}+{a}^{2} \right)   k_1 \right)\,  y_1^{3},
\\
q_3&= {b}^{4} \left( {a}^{2}-{c^2}k_1   \right) x_1^{2}+{a}^{4} \left( {b}^{2}+c^2k_1  \right)  y_1^{2}+2\,{a}^{2}{b}^{
2} c^2 k_2\, { x_1}\,{ y_1}.
\end{align*}

\noindent where:

\begin{align*}
k_1&=\frac{d_1^2\delta_1^2}{\,d_2}=\cos^2{\alpha},\\
k_2&=\frac{ \delta_1d_1^2}{d_2 }\sqrt{ d_2 -d_1^4\delta_1^2}=\sin{\alpha}\cos\alpha\\
c^2&=a^2-b^2,\;\; d_1=(a\,b/c)^2,\;\;d_2={b}^{4}x_1^2 +{a}^{4}y_1^2\\
\delta&=\sqrt{a^4+b^4-a^2 b^2},\;\;\delta_1=\sqrt{2 \delta-a^2-b^2}
\end{align*}

\noindent where $\alpha$, though not used here, is the angle of segment $P_1 P_2$ (and $P_1 P_3$) with respect to the normal at $P_1$.

\noindent The caustic is the ellipse:

\[\frac{x^2}{a''^2}+\frac{y^2}{b''^2}-1=0, \;\; a''=\frac{a( \delta-b^2)}{{a^2-b^2}}, \;\; b''=\frac{b(a^2-\delta)}{{a^2-b^2}}  \]

\subsection{N=4 Vertices \& Caustic}
\label{app:4-periodic}
\subsection{Simple}

The vertices of the 4-periodic orbit are given by:

\begin{align*}
   P_1=&(x_1,y_1),\;\;\;  P_2=\left(-\frac{ a^4y_1}{\sqrt{b^6x_1^2+a^6y_1^2}},\frac{ b^4x_1}{\sqrt{ b^6x_1^2+a^6y_1^2}} \right) \\
   P_3=&-P_1,\;\;\;\;\;\;\;\; P_4=  -P_2
\end{align*}

\noindent The caustic is the ellipse:

\[\frac{x^2}{a''^2}+\frac{y^2}{b''^2}-1=0, \;\; a''=\frac{a^2}{\sqrt{a^2+b^2}}, \;\; b''=\frac{b^2}{\sqrt{a^2+b^2}}  \]

\noindent The area and its bounds are given by:

\[A=\frac{2({b}^{4} x_1^2+{a}^{4}y_1^2 ) }{\sqrt {{b}^{6}x_1^2+ {a}^{6}y_1^2}}, \;\;\;\frac{4a^2b^2}{a^2+b^2}\leq A\leq 2ab\]

 The minimum (resp. maximum) area is achieved when the orbit is a  rectangle with $P_1=(x_1,b^2x_1/a^2)$ (resp. rhombus with $P_1=(a,0)$).

The perimeter is given by:

\[  L=4\sqrt{a^2+b^2}  \]

The exit angle $\alpha$ required to close the trajectory from a departing position $(x_1,y_1)$ on the billiard boundary is given by:

\begin{align*}
\cos{\alpha}=&{\frac {{a}^{2}b}{\sqrt {{a}^{2}+{b}^{2}}\sqrt {   {a}^{4}-c^2x_1^2}}}
\end{align*}

\subsubsection{Self-intersected}

When $a/b>\sqrt{2}$, the vertices of the 4-periodic self-intersecting orbit are given by:

\begin{align*}
    P_1 &= \left[au,b\sqrt{1-u^2}\right], \;\;\;   P_3  = \left[-au,b\sqrt{1-u^2}\right] \\
   P_2& =\left[-{\frac {a \sqrt{ {a}^{2}(a^2-2\, {b}^{2}) -   c^{4}{u}^
{2}} }{c^2\sqrt {1-{u}^{2} }  }},-{\frac {{b}^{3}}{c^2\sqrt {1-{u}^{2} }   }}\right]\\
 P_4& =\left[{\frac {a\sqrt{ {a}^{2}(a^2-2\, {b}^{2}) -   c^{4}{u}^
{2}} }{c^2\sqrt {1-{u}^{2} }  }},-{\frac {{b}^{3}}{c^2\sqrt {1-{u}^{2} }  }}\right]
\end{align*}

\noindent where $|u|{\leq}\frac{a}{c^2}
\sqrt {{a}^{2}-2\,{b}^{2}}$.

The confocal hyperbolic caustic is given by:

\[ \frac{x^2}{a''^2}-\frac{y^2}{b''^2}=1, \;\; a''= \frac{a\sqrt{a^2-2b^2}}{c}  \;\; b''=   \frac{ b^2}{c}\]

\noindent The four intersections of an ellipse with semi-axes $a,b$ with  confocal hyperbola with axes $a'',b''$ are given by:

\begin{equation}
\left[\pm  \frac {a a''}{c},\pm  \frac{b b''}{c} \right]
\label{eqn:ell-hyp-confocal}
\end{equation}

 
The exit angle $\alpha$ required to close the trajectory from a departing position $(x_1,y_1)$ on the billiard boundary is given by:
 
\begin{align*}
\cos{\alpha}=&
 \frac{a^2b}{ c \sqrt{a^4-c^2 x_1^2}}
\end{align*}

The perimeter of the orbit is 
$L=4a^2/c$.

\subsection{N=5 Vertices \& Caustic}
\label{app:5-periodic}
Let $a>b$ be the semi-axes of the elliptic billiard.

\begin{proposition}
The major semiaxis length $a''$ of the caustic for $N=5$ simple (resp. self-intersecting, i.e., pentagram) is given by
the 
root of the largest (resp. smallest) real root $x\in(0,a)$ of the following bi-sextic polynomial:

\begin{align*}
P_5(x)& = \;  {c}^{12}{x}^{12}-2\,{c}^{4}{a}^{2} \left( 3\,{a}^{8}-9\,{a}^{6}{b}^{2}
+31\,{a}^{4}{b}^{4}+{a}^{2}{b}^{6}+6\,{b}^{8} \right) {x}^{10}\\
&+{c}^{4}
{a}^{4} \left( 15\,{a}^{8}-30\,{a}^{6}{b}^{2}+191\,{a}^{4}{b}^{4}+16\,
{a}^{2}{b}^{6}+16\,{b}^{8} \right) {x}^{8}\\
&-4\,{c}^{4}{a}^{10} \left( 5
\,{a}^{4}-5\,{a}^{2}{b}^{2}+66\,{b}^{4} \right) {x}^{6}\\
& +{a}^{12}
 \left( 15\,{a}^{8}-30\,{a}^{6}{b}^{2}+191\,{a}^{4}{b}^{4}-368\,{a}^{2
}{b}^{6}+208\,{b}^{8} \right) {x}^{4}\\
&-2\,{a}^{14} \left( 3\,{a}^{8}-3
\,{a}^{6}{b}^{2}+22\,{a}^{4}{b}^{4}-48\,{a}^{2}{b}^{6}+32\,{b}^{8}
 \right) {x^2}+{a}^{24}
\end{align*}
\end{proposition}

\begin{proof}
 Consider a 5-periodic with vertices $P_i,i=1,...,5$ where $P_1$ is at $(a,0)$, i.e., the orbit is ``horizontal''. 
The polynomial $P_5$ is exactly the Cayley condition for the existence of 5-periodic orbits, see \cite{dragovic11}.
For $c=0$ the roots are $a_2^\prime=(\sqrt{5}-1)a/4$ and
$a_2^{\prime\prime}=(\sqrt{5}+1)a/4$ and corresponds to the regular case. For $b=0$ the roots are coincident in given by $x=a$. By analytic continuation, for $c\in (0,a)$, the two roots are in the interval $(0,a)$.
\end{proof}

For $N=5$ non-intersecting, the abcissae of vertices $P_2=(x_2,y_2)$, $P_3=(x_3,y_3)$ are given by the smallest positive solution (resp. unique negative) of the following equations:

\begin{align*}
    x_2:\;\;& {c}^{6}x_2^{6}-2\,a \left( 2\,{a}^{2}-{b}^{2} \right) {c}^{4}x_2^{5}+{a}^{2} \left( 5\,{a}^{2}+4\,{b}^{2} \right) {c}^{4}x_2^{4}-8\,{a}^{5}{b}^{2}{c}^{2}x_2^{3}\\&-{a}^{8} \left( 5\,
{a}^{2}-9\,{b}^{2} \right) x_2^{2}+2\,{a}^{9} \left( 2\,{a}^{2}
-{b}^{2} \right) {  x_2}-{a}^{12}
 = 0\\
  x_3:\;\;&{c}^{6}x_3^{6}-2\,a{b}^{2}{c}^{2} \left( 3\,{a}^{2}+{b}^{2}
 \right) x_3^{5} -{a}^{2} {c}^{2} \left( 3\,{a}^{4}-3\,{a}^{2}{b}^{2}+4
\,{b}^{4} \right) x_3^{4}+12\,a^5{b}^{2}{c}^{2}x_3^{
3} \\
& + {a}^{6}\left( 3\,{a}^{4}-3\,{a}^{2}{b}^{2}+4\,{b}^{4} \right) x_3^{2}-2\,{a}^{7}{b}^{2} \left( 3\,{a}^{2}-4\,{b}^{2}
 \right) {  x_3}-{a}^{12}
  =0
\end{align*}

The Joachimstall invariant $J$ of the simple orbit is given by the small positive root of 
\begin{align*}
& \;\;\; \; 4096 c^{12}J^{12}+2048  (3  a^2+b^2)  (a^2+3  b^2)  (a^2+b^2)  c^4  J^{10}\\
&-256  (29  a^4+54  a^2  b^2+29  b^4)  c^4  J^8+2304  (a^2+b^2) c^4  J^6\\
&-16  (3  a^2-4  a  b-3  b^2)  (3  a^2+4  a  b-3  b^2)  J^4-40(a^2+b^2)  J^2+5=0
\end{align*}
The perimeter of the simple orbit is given by
\begin{align*}
 L &=\frac{p}{q}\\
p&= \left( 1024\, \left( {a}^{2}+{b}^{2} \right) c^4{b}^{2}J^{7}-256\,c^4{b}^{2}J^{5}-64\,
 \left( {a}^{2}+{b}^{2} \right) {b}^{2}J^{3}+16\,J\,{b
}^{2} \right) \sqrt{1-4a^2J^2} \\
-&1024\, c^2
 \left( 5\,{a}^{4}+2\,{a}^{2}{b}^{2}+{b}^{4} \right) {b}^{2}J^{7}+256\,c^2 \left( 3\,{a}^{2}+{
b}^{2} \right) {b}^{2}J^{5}+64\, c^2 {b}^{2}J^{3}+16\,J\,{b}^{2}\\
q&= 
 256\, c^8J^{8}-
256\, c^2 \left( {a}^{2}+{b}^{2}
 \right)^{2}J^{6}+32\, c^2  \left( 3\,{a}^{2}+5\,{b}^{2} \right) J^{4}-16\,
c^2 J^{2}+1
\end{align*}


\subsection{N=6 Vertices \& Caustic}
\label{app:6-periodic}
\subsubsection{Simple}
Vertices $P_i,i=2,...,6$ with $P_1=(a,0)$ are given by: 
 
\begin{align*}
 P_4&=[-a,0]\\
 P_2&=[k_x,k_y],\;\;P_5=-P_2\\
 P_3&=[-k_x,k_y],\;\;P_6=-P_3\\
 k_x&=\frac{a^2}{a+b},\;\;k_y=\frac{b\sqrt{b (2\,a+b)}}{a+b}\\
\end{align*}

\noindent The confocal, elliptic caustic is given by:

\[ \frac{x^2}{a''^2} +\frac{y^2}{b''^2} =1,\;\;\;a''= \frac {a\sqrt {{a}(a+2\,b)}}{a+b}, b''= {\frac {b\sqrt {b \left( 2\,a+b \right) }}{a+b}}\]

\noindent The perimeter is given by:

\[ L={\frac {4({a}^{2}+ab+{b}^{2})}{a+b}}\]

\subsubsection{Self-Intersected (type I)}

This orbit only exists for $a>2b$. Vertices $P_i,i=2,...,6$ with $P_1=(0,b)$ are given by: 

\begin{align*}
 P_4&=[0,-b]\\
 P_2&=[k_x,k_y],\;\;P_5=-P_2\\
 P_3&=[k_x,-k_y],\;\;P_6=-P_3\\
 k_x&={\frac {a\sqrt {a \left( a-2\,b \right) }}{b-a}},\;\;k_y= {\frac {{b}^{2}}{b-a}}\\
\end{align*}

\noindent The confocal, hyperbolic caustic is given by:

\[ \frac{x^2}{a''^2} -\frac{y^2}{b''^2} =1,\;\;\;a''= \frac{{a}^{3/2}\sqrt{a-2\,b}}{a-b},\; b''= \frac {{b}^{3/2}\sqrt {2\,a-
b}}{a-b}\]

\noindent The 4 intersections of the above caustic with the EB are given by \eqref{eqn:ell-hyp-confocal}. 


\noindent The perimeter is given by:

\[ L=\frac {4({a}^{2}-ab+{b}^{2})}{a-b}\]

\subsubsection{Self-Intersected (type II)}

This orbit only exists for  $ a >\frac{2b\sqrt{3}}3$. Vertices $P_i,i=2,...,6$ with $P_1=[0,b]$ are given by:  
  are given by:
  \begin{align*}
 P_4&=[0,-b]\\
 P_2&=[k_x,k_y],\;\;P_3=-P_2 \\
 P_5&=[k_x,-k_y],\;\;P_6=-P_5\\
 k_x&= -\frac{{a}^{\frac{3}{2}} \sqrt { 2\,c-a}}{c}
 ,\;\;k_y=  {\frac { \left( c-a  \right) b
}{c}}
 \\
\end{align*}

The confocal hyperbolic caustic is given by:

\[ \frac{x^2}{a''^2} -\frac{y^2}{b''^2} =1,\;\;\;a''^2= {\frac {{a}^{3} \left( 3\,ac-2\,{b}^{2} \right) }{c \left( 3\,{a}^{2}+
{b}^{2} \right) }}
 ,\;b''^2= \frac {{b}^{2} \left( 2\,{a}^{2}(a-c)-{b}^{2}c \right) }{c
 \left( 3\,{a}^{2}+{b}^{2} \right) }
 \]
 
\noindent The 4 intersections of the above caustic with the EB are given by \eqref{eqn:ell-hyp-confocal}. 
 
 

 
\noindent The perimeter is given by:

\[ L=4(a+c)\sqrt{2a/c-1}\]

\subsection{N=7 Caustic}
\label{app:7-periodic}
Referring to Figure~\ref{fig:n7-three}, there are three types of 7-periodics: (i) non-intersecting, (ii) self-intersecting type I, i.e., with turning number 2, (iii) self-intersecting type II.

\begin{proposition}
The caustic semiaxis for non-intersecting 7-periodics (resp. self-intersecting type I, and type II self-intersecting) are given by the smallest  (resp. second and third smallest) root of the following degree-12 polynomial:

   
  {\small 
\begin{align*}
 &\; 
   {c}^{12}x_1^{12}- 4\left(  {a}^{2}+ {b}^{2} \right) {c}^{6}a
 \left( 3\,{a}^{2}+{b}^{2} \right) {b}^{2}x_1^{11}\\
 &-2\,{c}^{6}{a
}^{2} \left( 3\,{a}^{6}-6\,{a}^{4}{b}^{2}+13\,{a}^{2}{b}^{4}-2\,{b}^{6
} \right) x_1^{10}
+ \left( 60\,{a}^{4}+60\,{b}^{2}{a}^{2}+8\,{b
}^{4} \right) {c}^{6}{a}^{3}x_1^{9}\\
&+{a}^{6}   c^2 \left( 15
\,{a}^{8}-45\,{a}^{6}{b}^{2}+125\,{a}^{4}{b}^{4}-143\,{a}^{2}{b}^{6}+
112\,{b}^{8} \right)   x_1^{8}-8\,{a}^{7}{b}^{2}{c}^
{2} \left( 15\,{a}^{6}-20\,{a}^{4}{b}^{2}-7\,{a}^{2}{b}^{4}+8\,{b}^{6}
 \right) x_1^{7}\\
 &-4\,{a}^{8}{c}^{2} \left( 5\,{a}^{8}-10\,{a}^{6
}{b}^{2}+35\,{a}^{4}{b}^{4}-30\,{a}^{2}{b}^{6}+36\,{b}^{8} \right) x_1^{6}+8\,{a}^{9}{b}^{2}{c}^{2} \left( 15\,{a}^{6}-25\,{a}^{4}{b
}^{2}-2\,{a}^{2}{b}^{4}+4\,{b}^{6} \right) x_1^{5}\\
&+{a}^{10}{c}^
{2} \left( 15\,{a}^{8}-15\,{a}^{6}{b}^{2}+80\,{a}^{4}{b}^{4}-32\,{a}^{
2}{b}^{6}+64\,{b}^{8} \right) x_1^{4}-4\,{a}^{15}{b}^{2}
 \left( 15\,{a}^{4}-45\,{b}^{2}{a}^{2}+32\,{b}^{4} \right) x_1^
{3}\\
&-2\,{a}^{16} \left( 3\,{a}^{6}-3\,{a}^{4}{b}^{2}+10\,{a}^{2}{b}^{4}
-8\,{b}^{6} \right) x_1^{2}+4\,{a}^{17}{b}^{2} \left( 3\,{a}^{2
}-4\,{b}^{2} \right)  \left( {a}^{2}-2\,{b}^{2} \right)  x_1+{a}^{
24}=0
\end{align*}
}
\end{proposition}

It can shown the first two smallest (resp. third smallest) roots of the above polynomial are negative (resp. positive), and all have absolute values within $(0,a)$.

For $a=b$ the polynomial equation above is given by
${a}^{3}+4\,{a}^{2}x_1-4\,ax_1{2}-8\,x_1^{3}=0$ with roots
$ -0.9009688680 a,\; -0.2225209340a, \;0.6234898025 a. $

\subsection{N=8 Vertices \& Caustic}
\label{app:8-periodic}
\subsubsection{Simple}

Vertices $P_i,i=1,...,8$ with $P_1=[a,0]$ are given by: 
 \begin{align*}
     P_1 &= [a,0], \; \; P_5=-[a,0],\; \;P_3=[0,b], \; \;P_7=[0,-b]\\
     P_{2,4,6,8}&=[{\pm}a x,{\pm}b\sqrt{1-x^2}], \mbox{where $x$ is a positive root of:}\\
     &c^4 x^4-2a^2 c^2 x^3+2a^2b^2 x^2+2a^2 c^2 x-a^4=0
 \end{align*}
 
\subsubsection{Self-Intersected (type I and type II)}

These have hyperbolic confocal caustics; see Figures~\ref{fig:n8-si-I} and \ref{fig:n8-si-II}. Let $P_1=(x_1,y_1)$ be at the intersection of the hyperbolic caustic with the billiard for each case. For type I (resp. type II) $x_1$ is given by the smallest (resp. largest) positive root $x_1\in(0,a)$  of the following degree-8 polynomial:
 
 \begin{align*} 
 x_1:\;\;& {c}^{16}x_1^{8}-4\,{a}^{4}{c}^{8} \left( {a}^{6}-4\,{a}^{4}{b}^
{2}+{a}^{2}{b}^{4}-2\,{b}^{6} \right) x_1^{6} 
+2\,{a}^{8}{c}^{6}
 \left( 3\,{a}^{6}-15\,{a}^{4}{b}^{2}-4\,{b}^{6} \right) x_1^{4
}\\
&- 4\,{a}^{16}{c}^{4} \left( {a}^{2}-6\,{b}^{2} \right) x_1^{2}+
{a}^{20} \left( {a}^{4}-8\,{a}^{2}{b}^{2}+8\,{b}^{4} \right) =0
 \end{align*}
   

\subsubsection{Self-Intersected (type III)}

Let $P_1=(x_1,y_1)$ and $P_2=(x_1,-y_1)$ be two consecutive vertices in the doubled-up type III 8-periodic (dashed red in Figure~\ref{fig:n8-si-III}) connected by a vertical line (the figure is rotated, therefore this line will appear horizontal). $x_1^2$ is given by the smallest positive root of the following quartic polynomial on $\omega$:

\begin{align*}
   x_1:\;\;&
 \left( {a}^{4}+6\,{a}^{2}{b}^{2}+{b}^{4} \right) {c}^{4}x_1^{8
}-4\,{a}^{4} \left( {a}^{2}+5\,{b}^{2} \right) {c}^{4}x_1^{6}+2
\,{a}^{6} \left( 3\,{a}^{6}+6\,{a}^{4}{b}^{2}-21\,{a}^{2}{b}^{4}+16\,{
b}^{6} \right) x_1^{4}\\
& -4\,{a}^{8} \left( {a}^{6}+{a}^{4}{b}^{2}
-4\,{a}^{2}{b}^{4}+4\,{b}^{6} \right) x_1^{2}+{a}^{16}=0
\end{align*}

\begin{align*}
   x_1:\;\;&\alpha^{16}+\left(\alpha^2-1\right)^2 \left(\alpha^4+6 \alpha^2+1\right) \omega^4-4 \left(\alpha^2-1\right)^2 \left(\alpha^2+5\right) \alpha^4 \omega^3+\\
    &2 \left(3
   \left(\alpha^4+2 \alpha^2-7\right) \alpha^2+16\right) \alpha^6 \omega^2-4 \left(\alpha^6+\alpha^4-4 \alpha^2+4\right) \alpha^8 \omega=0
\end{align*}

\noindent where $\alpha=a/b$.

\section{Table of Symbols}
\label{app:symbols}
\begin{table}[H]
\begin{tabular}{|c|l|l|}
\hline
symbol & meaning \\
\hline
$O,N$ & center of billiard and vertex count \\
$L,J$ & inv. perimeter and Joachimsthal's constant \\
$a,b$ & billiard major, minor semi-axes \\
$a'',b''$ & caustic major, minor semi-axes\\
$f_1,f_2$ & foci \\
$P_i,P_i',P_i''$ &
$N$-periodic, outer, inner polygon vertices \\
$d_{j,i}$ & distance $|P_i-f_j|$ \\
$\kappa_i$ & ellipse curvature at $P_i$ \\
$\theta_i,\theta_i'$ & $N$-periodic, outer polygon angles \\
$A,A',A''$ &  $N$-periodic, outer, inner areas \\
\hline
$\rho$ & radius of the inversion circle \\
$P_{j,i}^{\dagger}$ & vertices of the inversive polygon wrt $f_j$ \\
$L_j^\dagger,A_j^\dagger$ & perimeter, area of inversive polygon wrt $f_j$\\
$\theta_{j,i}^\dagger$ & ith angle of inversive polygon wrt $f_j$ \\
\hline
\end{tabular}
\caption{Symbols used in the invariants. Note $i\in[1,N]$ and $j=1,2$.}
\label{tab:symbols}
\end{table}

\bibliographystyle{maa}
\bibliography{elliptic_billiards_v3,authors_rgk_v1} 

\end{document}